\documentclass[preprint,12pt]{elsarticle}
\usepackage[margin=1in]{geometry}

\usepackage{amssymb}
\usepackage{amsmath}
\usepackage{amsthm}

\usepackage{hyperref}
\usepackage{cleveref}
\usepackage{mathrsfs}
\usepackage{subcaption}
\usepackage{tikz-cd} 
\usetikzlibrary{plotmarks}
\usetikzlibrary{arrows,automata}
\usepackage{tikz-network}
\usepackage{xcolor}
\usepackage{url}
\usepackage{multicol}
\usepackage{mathtools}

\def\sym{\mathrm{Sym}}
\def\G{\mathcal{G}}
\def\L{\mathscr{L}}
\def\I{\mathscr{I}}
\def\R{\mathbb{R}}
\def\P{\mathbb{P}}
\def\lca{\mathrm{lca}}
\def\pa{\mathrm{pa}}
\def\Int{\mathrm{Int}}

\def\C{\mathbb{C}}
\def\T{\mathcal{T}}
\def\desc{\mathrm{desc}}

\def\Int{\mathrm{Int}}

\def\Lv{\mathrm{Lv}}
\def\cL{\mathcal{L}}

\def\height{\mathrm{ht}}

\def\rZ{\mathrm{Z}}
\def\CC{\mathcal{C}}

\definecolor{darkgreen}{rgb}{0,0.8,0}
\definecolor{darkyellow}{rgb}{0.95,0.85,0}

\newtheorem{theorem}{Theorem}
\newtheorem{remark}[theorem]{Remark} 
\newtheorem{example}[theorem]{Example} 
\newtheorem{conjecture}[theorem]{Conjecture} 
\newtheorem{corollary}[theorem]{Corollary}
\newtheorem{lemma}[theorem]{Lemma}
\newtheorem{proposition}[theorem]{Proposition}

\begin{document}

\begin{frontmatter}



\title{Toric Multivariate Gaussian Models from Symmetries in a Tree}


\author[1]{Emma Cardwell}
\ead{emma_cardwell@brown.edu}

\author[3,4,5]{Aida Maraj%
}
\ead{maraj@mpi-cbg.de}

\author[2]{\'Alvaro Ribot%
}
\ead{aribotbarrado@g.harvard.edu}

\affiliation[1]{Brown University}
\affiliation[2]{Harvard University}
\affiliation[3]{Max Planck Institute of Molecular Cell Biology and Genetics}
\affiliation[4]{Center for Systems Biology Dresden}
\affiliation[5]{Technical University Dresden}

\begin{abstract}
Given a rooted tree~$\mathcal{T}$ on~$n$ non-root leaves with colored and zeroed nodes, we construct a linear space~$L_\mathcal{T}$ of~$n\times n$ symmetric matrices with constraints determined by the combinatorics of the tree. When~$L_\mathcal{T}$ represents the covariance matrices of a Gaussian model, it provides natural generalizations of Brownian motion tree (BMT) models in phylogenetics and a step toward a more accurate model for phylogenetic networks with symmetries for species hybridization. When~$L_\mathcal{T}$ represents a space of concentration matrices of a Gaussian model, it gives certain colored Gaussian graphical models, which we refer to as BMT derived models. We investigate conditions under which the reciprocal variety~$L_\mathcal{T}^{-1}$ is toric. Relying on the birational isomorphism of the inverse matrix map, we show that if the BMT derived graph of~$\mathcal{T}$ is vertex-regular and a block graph,  under the derived Laplacian transformation, which we introduce, ~$L_\mathcal{T}^{-1}$ is the vanishing locus of a toric ideal. This ideal is given by the sum of the toric ideal of the Gaussian graphical model on the block graph, the toric ideal of the original BMT model, and binomial linear conditions coming from vertex-regularity. To this end, we provide monomial parametrizations for these toric models realized through paths among leaves in~$\mathcal{T}$.
\end{abstract}



\end{frontmatter}



\section{Introduction}

A multivariate Gaussian distribution of a random vector $\mathbf{X}=(X_1,\ldots,X_n)$ is determined by its mean and its \emph{covariance matrix}, which is an~$n {\times} n$ symmetric positive definite matrix. 
A Gaussian model is a {family} of Gaussian distributions feasible for~$\mathbf{X}$. By setting the mean equal to zero, we can identify the Gaussian model with its parameter space of covariance matrices or, alternatively, the space of the inverses of covariance matrices, known as \emph{concentration} or \emph{precision matrices}. This paper considers linear spaces of symmetric matrices determined by a tree with colored and zeroed nodes. We investigate when its reciprocal space is a toric variety; that is, it is isomorphic to the solution set of a toric ideal.

In applications, a wide class of Gaussian models is represented by linear conditions in a symmetric matrix. Interpreted as linear conditions in the covariance matrices, these give natural generalizations of \emph{Brownian motion tree (BMT) models} in phylogenetics. BMT models,  first introduced by Felsenstein \cite{felsenstein1973maximum}, are used to test selective pressure \cite{cooper2010body,freckleton2006detecting},  to represent continuous molecular traits \cite{brawand2011evolution}, to serve as a null model of evolution under genetic drift
\cite{schraiber2013inferring}, and in other applications outside biology 
 \cite{eriksson2010toward,tsang2004network}. Recent advances in the representation of BMT models as toric varieties \cite{sturmfels2019brownian} have led to progress in maximum likelihood estimation of these models
\cite{boege2021reciprocal,coons2024maximum,truell2022maximumlikelihoodestimationbrownian}. When treated as linear conditions on the concentration matrices, our models are types of \emph{colored Gaussian graphical (CGG) models}, which we will refer to as BMT-derived models. Colored Gaussian graphical models, introduced by Hojsgaard and Lauritzen in \cite{hojsgaard2008}, are more detailed versions of Gaussian graphical models, used to study gene regulatory networks where symmetries are imposed among genes with similar expression patterns~\cite{bibby1979multivariate,toh2002,vinciotti2016,wit2015},  to analyze longitudinal data on the performance of several companies in the same market \cite{abbruzzo2016}, and to reduce the maximum likelihood threshold of the model
~\cite{MRS21,uhler2011}. They have a rich algebraic structure, which changes with small modifications on the symmetries/colors in the graph \cite{coons2023symmetrically,davies2021coloured,uhler2011}.  Lastly, our BMT-derived models are a step toward a more accurate model for phylogenetic networks with symmetries for species hybridization, which we discuss in the last section of the paper.

A phylogenetic tree~$T$ is a rooted directed tree such that one of the leaves serves as its root and all edges are directed away from the root. We label the root with~$0$ and the non-root leaves with~$1,\ldots,n$. The set~$\Lv(T)$ of non-root leaves represents species of interest while the set of internal nodes~$\Int(T)$ represent their  common ancestors. We denote the least common ancestor of two non-root leaves~$i,j$ with~$\lca(i,j)$. A \emph{colored phylogenetic tree with zeroed nodes} is a phylogenetic tree with colored and potentially zeroed non-root nodes. 
We use~$\lambda_{\T}(i)$ to denote the color of node~$i$ in~$\T$,~$\Lambda_\T = \{ \lambda_\T (i) \mid i \in \Lv(\T) \cup \Int(\T)\}$ is the set of colors in~$\T$, and~$\rZ(\T) \subseteq \Int(\T)$ is the set of zeroed nodes in~$\T$. The only assumption we make on the coloring is that leaves and internal nodes do not share colors. Throughout this article, we will consistently use~$T$ for uncolored trees and~$\T = (T,\Lambda_{\T},\rZ(\T))$ for colored trees.

Let us describe the linear model associated to~$\T$. Assign to a color~$\lambda\in\Lambda_\T$,  the parameter~$t_{\lambda}$. Let $\sym_n(\R)$ denote the vector space of symmetric~$n\times n$ real-valued matrices and let~$\mathrm{PD}_n$ denote the space of~$n\times n$ real-valued positive definite matrices. The \emph{linear Gaussian model for~$\T$} is 
\[
\resizebox{0.99\linewidth}{!}{
$L_\T=\{M\in \sym_n(\R)\mid  M_{ij}=0 \text{ if } \lca(i,j)\in \rZ(\T)    \text{ and }    M_{ij}=t_{\lambda_\T(\lca(i,j))}  \text{ else} \}.$}
\]
The \emph{reciprocal (or inverse) space} of~$L_\T$ is~$L_\T^{-1} = \overline{\{\Sigma\in\sym_n(\R) \mid \Sigma^{-1}\in L_\T \}}$, where the overline denotes Zariski closure. 
 The \emph{vanishing ideal of the model} is the set of equations vanishing on~$L_\T^{-1}$.

The \emph{BMT model induced by tree~$\T$ with colored and zeroed nodes} is the set of multivariate Gaussian distributions with mean zero and set of covariance matrices~$L_\T\cap \mathrm{PD}_n$. Therefore,~$L_\T^{-1}\cap \mathrm{PD}_n$ is the set of concentration matrices for this model. 

Reversing the roles of covariance and concentration matrices,~$L_\T\cap \mathrm{PD}_n$ is the set of concentration matrices a class of {colored Gaussian graphical models} which we name BMT derived models. This is not straightforward. We recall the definition of a CGG model here and discuss the CGG models derived by~$L_\T$ in detail in \Cref{sec: combinatorics}.  A colored graph is a tuple~$\G = (G, \Lambda_\G)$ where~$G=(V,E)$ is a simple graph and~$\Lambda_\G$ is a set of colors used for edges and vertices of~$G$. The only restriction is that vertices and edges do not share colors. 
The CGG model for~$\G$ is the multivariate Gaussian model with mean zero and set of concentration matrices~$\L_\G\cap \mathrm{PD}_n$, where~$\L_{\G}$ is the linear space of symmetric matrices~$K = (k_{ij}) \in \sym_n(\R)$ satisfying:
\begin{enumerate}
    \item~$k_{ii}= k_{jj}$ if~$\lambda_\G(i) = \lambda_\G(j)$ for~$i,j\in V(\G)$,
    \item~$k_{ij} = k_{\ell m}$ if~$\lambda_\G(\{i,j\}) = \lambda_\G(\{\ell,m\})$ for~$\{i,j\},\{\ell,m\}\in E(\G)$, and
    \item~$k_{ij} = 0$ if~$\{i,j\} \not\in E(\G)$.
\end{enumerate}
Given~$\T$, consider the colored graph~$\G$ such that~$L_T=\L_\G$, as described in \Cref{prop:graph_from_colored_zeroed}. We call this the BMT derived model. An illustration is given in \Cref{fig:intro}.  

\setcounter{figure}{0}

\begin{figure}[h]
    \centering
    \begin{subfigure}[b]{0.3\textwidth}
        \centering
        \begin{tikzpicture}[scale=0.4]
            \tikzset{
                VertexStyle/.style={
                    shape=circle,
                    draw,
                    fill=white,
                    minimum size=10pt,
                    inner sep=0pt,
                    font=\tiny
                },
                EdgeStyle/.style={
                    color=black,
                    line width=0.5pt
                }
            }
            \Vertex[x=0, y=-0.5, label=0, color=white, style={opacity=0}]{n0}
            \Vertex[x=-1, y=-6, label=\textcolor{black}{2}, color=cyan]{n2}
            \Vertex[x=-3, y=-6, label=\textcolor{black}{1}, color=cyan]{n1}
            \Vertex[x=1, y=-6, label=3, color=yellow]{n3}
            \Vertex[x=3, y=-6, label=\textcolor{white}{4}, color=darkgreen]{n4}
            \Vertex[x=-2, y=-4.66, label=\textcolor{white}{5}, color=red]{n5}
            \Vertex[x=-1, y=-3.33, label=6, color=white]{n6}
            \Vertex[x=0, y=-2, label=\textcolor{white}{7}, color=blue]{n7}

            \Edge[color=black](n0)(n7)
            \Edge[color=black](n7)(n6)
            \Edge[color=black](n6)(n5)
            \Edge[color=black](n7)(n4)
            \Edge[color=black](n5)(n2)
            \Edge[color=black](n5)(n1)
            \Edge[color=black](n6)(n3)
        \end{tikzpicture}
        \caption{Phylogenetic tree~$\T$.}
    \end{subfigure}
    \begin{subfigure}[b]{0.3\textwidth}
        \centering
        \[
        \begin{bmatrix}
        \color{cyan}{t_c} & \color{red}{t_r} & 0 & \color{blue}{t_b} \\
        \color{red}{t_r} & \color{cyan}{t_c} &  0  &  \color{blue}{t_b}  \\
        0  &  0  & \color{yellow}{t_y} & \color{blue}{t_b} \\
        \color{blue}{t_b}  &  \color{blue}{t_b}  & \color{blue}{t_b}  & \color{darkgreen}{t_g}
        \end{bmatrix}
        \]
        \caption{Associated matrix in~$L_\T$.}
    \end{subfigure}
    \begin{subfigure}[b]{0.3\textwidth}
        \centering
        \begin{tikzpicture}[scale=1]
            \Vertex[label=$\textcolor{black}{1}$,color=cyan]{A};
            \Vertex[x=2,label=$\textcolor{black}{2}$,color=cyan]{B};
            \Vertex[y=-2,label=$\textcolor{white}{4}$,color=darkgreen]{D};
            \Vertex[x=2,y=-2, label=$3$,color=yellow]{C};
            \Edge[color=red](A)(B);
            \Edge[color=blue](A)(D);
            \Edge[color=blue](B)(D);
            \Edge[color=blue](C)(D);
        \end{tikzpicture}
        \caption{Derived colored graph~$\G$.}
    \end{subfigure}
    \caption{A colored and zeroed phylogenetic tree~$\T$ with zeroed internal node~$6$, its associated matrix in~$L_\T$, and colored graph~$\G$ derived from~$\T$. The reciprocal variety~$L_\T^{-1}=\cL^{-1}_\G$ is toric under the~$G$-derived Laplacian (\ref{eqn:derived_laplacian}), with toric vanishing ideal}
    \vspace{-0.5cm}
    \[
    \langle p_{14}-p_{24},p_{13}-p_{23},p_{01}-p_{02},p_{03}p_{24}-p_{02}p_{34},p_{04}p_{23}-p_{24}p_{34}\rangle.
    \]
    \vspace{-0.5cm}
    \label{fig:intro}
\end{figure}
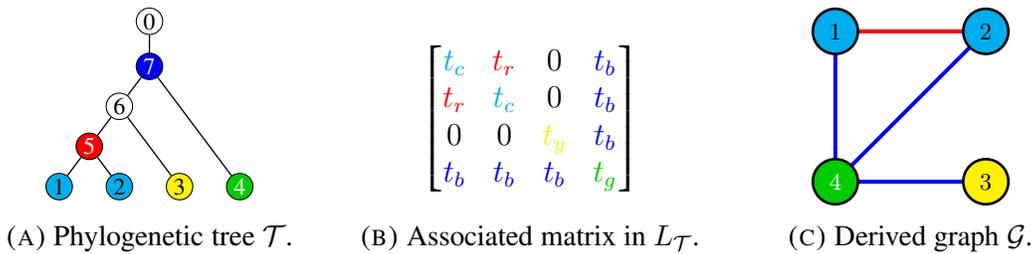
 Unlike the linear structure of~$L_\T$, the reciprocal variety~$L_\T^{-1}$ is generally defined by nonlinear conditions. This results in a complex geometry, which can drastically change with small changes on the coloring or zeroing in the tree. 
 Our goal is to identify combinatorial conditions on~$\T$, or on its associated BMT derived graph~$\G$, for which~$L_\T^{-1}=\L_\G^{-1}$ is a toric variety.  

Our work is motivated by recent advances on BMT models, that is,  when the tree has no zeroed nodes and each node has its own color.  Sturmfels, Uhler and Zwiernik \cite{sturmfels2019brownian} show that under the reduced graph Laplacian transformation, the reciprocal variety is the zero set of a toric ideal. The binomial equations and its monomial parametrization \cite{boege2021reciprocal} have been instrumental in deriving results on the maximum likelihood estimate \cite{coons2024maximum,truell2022maximumlikelihoodestimationbrownian}, and the dual maximum likelihood estimate~\cite{boege2021reciprocal} for BMT models. 
 
We informally state the main result of the paper in \Cref{thm: main}.  For ease of readability, we introduce the necessary notation. Let~$\T=(T,\Lambda_\T,Z(\T))$ be a tree with colored and zeroed nodes, and let~$\G=(G,\Lambda_\G)$ be its BMT-derived graph as in \Cref{prop:graph_from_colored_zeroed}. Let~$I_T$ be the toric vanishing ideal for~$L_T^{-1}$ under the reduced graph Laplacian transformation, stated in \Cref{thm:brownian_toric}. Let~$\I_G$ be the vanishing ideal for the Gaussian graphical model determined by the uncolored graph~$G$; that is, the vanishing ideal of~$\L_G^{-1}$. When~$G$ is a block graph,~$\I_G$ is known to be toric, with its binomials described in \Cref{thm:misra_sullivant}. Lastly, let~$\I_{\overline{\G}}$ be the ideal generated by linear forms~$\sigma_{ik} - \sigma_{jk}$ for any vertices~$i,j$ of~$\G$ that share the same color. Treat all the toric ideals~$I_T,\I_\G,\I_{\overline{G}}$ as over the same polynomial ring.

\begin{theorem}\label{thm: main} 
 Let~$\T=(T,\Lambda_\T,Z(\T))$  be a colored phylogenetic tree with zeroed nodes and let~$\G = (G, \Lambda_\G)$ be its BMT-derived graph. If~$\G$ is a vertex-regular block graph, then~$L_\T^{-1}=\L_\G^{-1}$ is a toric variety with \(\sqrt{I_T+ \I_G + \I_{\overline{\G}}}\) its vanishing ideal under the linear transformation given by the~$G$-derived Laplacian transformation (\Cref{sec:g-derived-laplacian}). Monomial parametrizations are provided in (\ref{eqn:generalized path map}) and (\ref{eqn:generalized path map block}).
\end{theorem}
 
The proof is done in three parts. First, we show the case where no nodes are zeroed, so the graph~$\G$ is complete and vertex-regular. Second, we show the case where no color constraints are imposed on~$\T$ and~$\G$ is a block graph. Lastly, we combine the two steps. 
Each of these steps imply further linear conditions on the set of covariance matrices. That is,~$L_\T$ is an intersection of linear spaces. Since matrix inversion is a birational isomorphism, the variety~$L_\T^{-1}$ is the intersection of the reciprocal varieties of the aforementioned linear spaces. We show that this intersection is also a toric variety.\\

\textbf{Structure of the paper.} In \Cref{sec: binomial}
we discuss aspects of binomial and toric ideals relevant to our work and deduce their properties under birational isomorphisms. These properties are critical for proving binomiality results and can also serve as tests for toricness in similar instances, without the need to explicitly find a generating set. 
In \Cref{sec: prelim}, we outline preliminaries on toric ideals, Brownian motion tree models, and colored Gaussian graphical models. In \Cref{sec: combinatorics}, we relate properties of a colored phylogenetic tree to the combinatorial properties of its associated colored graphical model. \Cref{sec:4}, \Cref{sec:5}, and \Cref{sec:6} are dedicated to proving \Cref{thm: main}. We end with a discussion of applications to phylo-genetic networks for species hybridization. 

Macaulay2 \cite{M2} code for the examples shared in this paper, as well as supporting conjectures, is available on GitHub: \url{https://github.com/esc1734/colored-zeroed-trees}.

\section{Intersections of toric varieties and birational isomorphisms} \label{sec: binomial}  
Given an algebraic variety~$V$, its vanishing radical ideal~$I(V)$ in a Noetherian polynomial ring~$R$ is the set of polynomials~$f\in R$ such that~$f(v)=0$ for all~$v\in V$. If~$V$ is an irreducible variety, then~$I(V)$ is a prime ideal. 
An ideal is \emph{toric} if and only if it is prime and generated by binomials. Equivalently, an ideal is toric if and only if it is the kernel of a monomial rational map. 

In the next sections, we will see that our varieties of interest are finite intersections of toric varieties, and we will need to conclude that this intersection is again a toric variety. This is not true in general, since the intersection of irreducible varieties is often reducible. However, the radical of a binomial ideal is a binomial ideal, which implies the following.

\begin{lemma}
[{\cite[Theorem 3.1]{eisenbud1996binomialideals}}]   \label{thm:eisenbud_sturmfels_binomial} Let~$V_1,V_2\subset \P^n$ have binomial vanishing ideals $I(V_1)$ and~$I(V_2)$. Then,~$V_1\cap V_2$ has binomial vanishing ideal~$\sqrt{I(V_1)+I(V_2)}$.
\end{lemma}

We seek conditions for when this radical is a prime ideal, given that~$V_1$ and~$V_2$ are irreducible. To do so, we use birational isomorphisms.  Given two varieties~$V,W$, a rational map~$\rho: V \dashrightarrow W~$ is a \emph{birational isomorphism} if there exists some rational map~$\rho^{-1}$ such that \(\rho \circ \rho^{-1}\) and \(\rho^{-1} \circ \rho\) are defined and agree with the identity map on an open dense subset of each variety. {Of particular relevance is the matrix inversion map $\rho_M$, defined as}
\begin{equation} \label{eqn:general inverse_map}
    \begin{array}{cccc}
        \rho_M: & \C[\sigma_{ij}\mid 1\leq i\leq j\leq n] &\to & \C(m_{ij} \mid  1\leq i\leq j\leq n) \\
 & \sigma_{ij} &\mapsto  &\frac{(-1)^{i+j}M_{[n]\setminus \{i\},[n]\setminus\{j\}}}{\det (M)}.
    \end{array}
\end{equation}
{Observe that $\rho_M \circ \rho_M$ is defined and agrees with the identity over the set of invertible matrices, so $\rho_M$ is a birational isomorphism.} Our varieties of interest are the closures of the images of linear spaces of symmetric matrices under~$\rho_M$. The following shows that birational maps are useful for computing the intersection of algebraic varieties.

\begin{lemma}\label{lemma:birational}
 Let~$\rho:\P^n\to \P^m$ be a birational isomorphism and let~$V_1,V_2$ be algebraic varieties on~$\P^{n}$ such that the restriction of~$\rho$ to~$V_1, V_2, V_1 \cap V_2$ is still birational. Then,~$\rho(V_1\cap V_2) = \rho(V_1) \cap \rho(V_2)$. In particular, \[I({\rho(V_1\cap V_2)}) = \sqrt{I({\rho(V_1)}) + I({\rho(V_2)})}.\]
\end{lemma}

\begin{proof}
Since~$\rho$ is a birational isomorphism, we have open dense subsets~$S_1\subseteq V_1$, $S_2\subseteq V_2$ such that~$\rho^{-1}\circ \rho$ is the identity when restricted to~$S_1,S_2$. By definition,~$\rho(V_i) = \overline{\rho(S_i)}$. In particular,~$\rho$ is injective when restricted to~$S_1 \cup S_2$. Therefore,~$\rho(S_1 \cap S_2) = \rho(S_1) \cap \rho (S_2)$. Moreover,~$\rho(S_1), \rho(S_2)$ are open dense subsets of~$\rho(V_1),\rho(V_2)$, respectively. Then, by definition of subspace topology,~$\rho(S_i) \cap (\rho(V_1) \cap \rho(V_2))$ is an open dense subset in~$\rho(V_1) \cap \rho(V_2)$. Finally, since the intersection of open dense subsets is dense, we get
$\rho(V_1\cap V_2) = \overline{\rho(S_1\cap S_2)} = \overline{\rho(S_1)\cap \rho(S_2)} = \rho(V_1) \cap \rho(V_2)
$.
By Hilbert's Nullstellensatz, we have~$I\left(\rho(V_1)\cap\rho(V_2)\right)= \sqrt{I(\rho(V_1))+ I(\rho(V_2))}$.
\end{proof}

The above combines to give the following theorem, which will be applied later to the inverse map~(\ref{eqn:general inverse_map}) to prove our main results. 

\begin{theorem}\label{thm:ACTUAL_INTERSECTION}
Let~$\rho:\P^n\rightarrow \P^m$ be a birational isomorphism.   Let~$V_1, V_2\subseteq \P^n$ be irreducible algebraic varieties such that~$V_1 \cap V_2$ is also irreducible. Suppose that there is an invertible linear change of variables~$F$ such that both~$I(\rho(V_1))$ and~$I(\rho(V_2))$ are toric ideals under~$F$. Then,~$\rho(V_1 \cap V_2)$ is a toric variety with vanishing ideal~$\sqrt{I(\rho(V_1))+I(\rho(V_2))}$ which is toric under~$F$.    
\end{theorem} 

\begin{proof}
 By \Cref{lemma:birational},~$\sqrt{I(\rho(V_1))+I(\rho(V_2))}$ is the vanishing ideal for~$\rho(V_1 \cap V_2)$. Since~$V_1 \cap V_2$ is irreducible,~$\rho(V_1 \cap V_2)$ is also irreducible. Hence, the ideal $\sqrt{I(\rho(V_1))+I(\rho(V_2))}$ is  prime. Applying an invertible linear transformation preserves irreducibility. Hence,~$\sqrt{I(\rho(V_1))+I(\rho(V_2))}$ under~$F$ is a prime ideal.

Since the ideals~$I(\rho(V_1)), I(\rho(V_2))$ are binomial ideals under~$F$, their sum is also binomial under~$F$. The radical of their sum is binomial under~$F$, by \Cref{thm:eisenbud_sturmfels_binomial}. We conclude that~$\rho(V_1\cap V_2)$ is a toric variety with vanishing ideal~$\sqrt{I(\rho(V_1))+I(\rho(V_2))}$. 
\end{proof}

\begin{remark}
    The results in this section involve projective varieties, but they also apply to the corresponding affine cones. In particular, this applies to the affine varieties we study in this paper, such as~$L_\T^{-1}$ and~$\L_\G^{-1}$.
\end{remark}

\section{Introducing BMT-derived models} \label{sec: prelim}

Here we review the main results in the literature on toric geometry for BMT models and CGG models. Then we define CGG models derived from a colored and zeroed tree, which we refer to as BMT-derived models. In \Cref{prop:graph_from_colored_zeroed}, we show how all BMT-derived models can be realized by operations of color merging and deletion of vertices and edges of certain colored complete graphs. We end this section by introducing the generalized path map on a tree with colored and zeroed nodes. This monomial map describes the reciprocal variety when it is toric. 

\subsection{Literature on Brownian motion tree (BMT) models} BMT models are multivariate Gaussian models derived from a phylogenetic tree~$T$ when no node is zeroed and no additional symmetries are added, i.e., each node has its own color. Given a tree~$T$ with~$n$ non-root leaves, its space of covariance matrices is~$L_T\cap \mathrm{PD}_n$.

Consider the linear change of variables on~$\C[p_{ij} \mid 0\leq i < j\leq n]$ given by
\begin{align}
\label{eqn:graph_laplacian_BMT}
\begin{aligned}
    p_{ij} &= -\sigma_{ij} \quad &\textrm{for }&1\leq i< j\leq n,\\
    p_{0i} &= \sum_{j=1}^n \sigma_{ij} \quad &\textrm{for }&1\leq i\leq n.
\end{aligned} 
\end{align}
{We call this the \emph{reduced graph Laplacian transformation}. The matrix of the transformation is the {reduced graph Laplacian}} for the complete graph on~$n+1$ vertices with edge weights~$p_{ij}$, and with the first column and row removed. 

\begin{theorem}[{\cite[Theorem 1.2]{sturmfels2019brownian}}] \label{thm:brownian_toric}
 Let~$T$ be a phylogenetic tree. Consider the toric ideal~$I_T$ generated by the binomials~$p_{ik}p_{j\ell} - p_{i\ell} p_{jk}$ for cherries~$\{i,j\}, \{k,\ell\}$ in the induced 4-leaf subtree on any quadruple of leaves~$i,j,k,\ell \in \{0\}\cup \Lv(\T)$.  Then,~$I_T$ is the vanishing ideal of~$L_T^{-1}$ under the reduced graph Laplacian transformation (\ref{eqn:graph_laplacian_BMT}).
\end{theorem}
By cherries, we mean a pair of leaves adjacent to the same internal node.
The authors in \cite{boege2021reciprocal} connect this ideal to the paths in the tree. Let~$E(\T)$ be the set of edges of~$\T$, and let~$i \leftrightsquigarrow j$ denote the set of edges in the path that connects leaf~$i$ to leaf~$j$. Direct each edge in~$T$ to point away from the root. Associate to each directed edge~$(\ell,k)$ a parameter~$\theta_k$; see \Cref{fig:uncolored}.

\begin{figure}[h]
    \centering
    \begin{subfigure}[b]{0.3\textwidth}
        \centering
        \begin{tikzpicture}[scale=0.4]
            \tikzset{
                VertexStyle/.style={
                    shape=circle,
                    draw,
                    fill=white,
                    minimum size=10pt,
                    inner sep=0pt,
                    font=\tiny
                },
                EdgeStyle/.style={
                    color=black,
                    line width=0.5pt
                }
            }
            \Vertex[x=0, y=6, label=0, color=gray!0, style={opacity=0}]{n0}
            \Vertex[x=-1, y=0, label=\textcolor{black}{2}, color=yellow]{n2}
            \Vertex[x=-3, y=0, label=1, color=cyan]{n1}
            \Vertex[x=1, y=0, label=3, color=magenta]{n3}
            \Vertex[x=3, y=0, label=4, color=violet!50]{n4}
            \Vertex[x=-2, y=2, label=\textcolor{white}{5}, color=red]{n5}
            \Vertex[x=2, y=2, label=\textcolor{white}{6}, color=darkgreen]{n6}
            \Vertex[x=0, y=4, label=\textcolor{white}{7}, color=blue]{n7}

            \draw (n0) -- (n7) [black] node [pos=0.5, right, black] {\tiny~$\theta_7$};
            \draw (n7) -- (n6) [black] node [pos=0.3, right, black] {\tiny~$\theta_6$};
            \draw (n6) -- (n4) [black] node [pos=0.3, right, black] {\tiny~$\theta_4$};
            \draw (n7) -- (n5) [black] node [pos=0.3, left, black] {\tiny~$\theta_5$};
            \draw (n5) -- (n2) [black] node [pos=0.3, right, black] {\tiny~$\theta_2$};
            \draw (n5) -- (n1) [black] node [pos=0.3, left, black] {\tiny~$\theta_1$};
            \draw (n6) -- (n3) [black] node [pos=0.3, left, black] {\tiny~$\theta_3$};
        \end{tikzpicture}
        \caption*{}
    \end{subfigure}
    \begin{subfigure}[b]{0.3\textwidth}
        \centering
        \[
         \begin{bmatrix}
                {\color{cyan}{t_1}} & {\color{red}{t_5}} & {\color{blue}{t_7}} & {\color{blue}{t_7}}\\
               {\color{red}{t_5}} & {\color{yellow}{t_2} }&{\color{blue}{t_7}}& {\color{blue}{t_7}}\\
               { \color{blue}{t_7} }& {\color{blue}{t_7}}& {\color{magenta}{t_3}} & {\color{darkgreen}{t_6}}\\
               {\color{blue}{t_7}} &{\color{blue}{t_7}} &  {\color{darkgreen}{t_6}} & {\color{violet!50}{t_4}}
            \end{bmatrix}
        \]
        \caption*{}
    \end{subfigure}
    \begin{subfigure}[b]{0.3\textwidth}
        \centering
        \begin{tikzpicture}[scale=1]
            \Vertex[label=$1$,color=cyan]{A};
            \Vertex[x=2,label=\textcolor{black}{$2$},color=yellow]{B};
            \Vertex[y=-2,label=$4$,color=violet!50]{D};
            \Vertex[x=2,y=-2, label=$3$,color=magenta]{C};
            \Edge[color=red](A)(B);
            \Edge[color=blue](A)(D);
            \Edge[color=blue](B)(D);
            \Edge[color=darkgreen](C)(D);
            \Edge[color=blue](A)(C);
            \Edge[color=blue](B)(C);
        \end{tikzpicture}
        \caption*{}
    \end{subfigure}
    \vspace{-0.7cm}
    \caption{A binary tree on 4 non-root leaves and its corresponding colored graph. The tree can be thought of as an uncolored tree, since all nodes have different colors.}
    \label{fig:uncolored}
\end{figure}
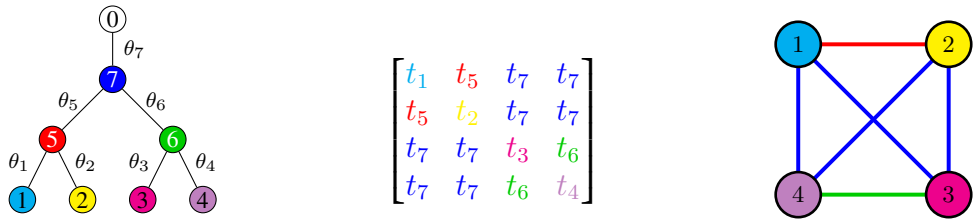
 
\begin{theorem}[{\cite[Proposition 3.1]{boege2021reciprocal}}] \label{thm:brownian_monomial map}
Let~$T$ be a phylogenetic tree. Then, the ideal~$I_T$ from \Cref{thm:brownian_toric} is the kernel of the path map 
 \begin{align} \label{eqn:shortest path map}
\varphi_T: \C[p_{ij} \mid 0\leq i<j\leq n] \to \C[\theta_k\mid k\in V(\T)], \ 
p_{ij} \mapsto \prod_{(\ell, k)\in i \leftrightsquigarrow j} \theta_k. 
\end{align} 
\end{theorem}

\begin{example}
Let~$T$ be the tree from \Cref{fig:uncolored}.
Using Macaulay2, we find that the vanishing ideal under the reduced graph Laplacian transformation is
\[
 \resizebox{0.99\linewidth}{!}{
$I_T= (p_{14}p_{23}-p_{13}p_{24},p_{04}p_{23}-p_{03}p_{24},p_{02}p_{14}-p_{01}p_{24},p_{04}p_{13}-p_{03}p_{14},p_{02}p_{13}-p_{01}p_{23}).$}
\]
This is the kernel of $\varphi_T : \R[p_{ij}\mid 0\leq i<j\leq 4]\rightarrow \R[\theta_1,\theta_2,\theta_3,\theta_4,\theta_5,\theta_6,\theta_7]$ given by~$\varphi_T(p_{01})=\theta_1\theta_5\theta_7, \varphi_T(p_{02})=\theta_2\theta_5\theta_7, \varphi_T(p_{03})=\theta_3\theta_6\theta_7, \ldots, \varphi_T(p_{24})=\theta_2\theta_5\theta_6\theta_4, \varphi_T(p_{34})=\theta_3\theta_4$.
\end{example} 

\subsection{Literature on Gaussian graphical models}

The main results on the geometry of Gaussian graphical models assume $G$ is a block graph. A graph $G$ is a \textit{$c$-clique sum} of smaller graphs~$G_1,G_2$ if there exists a partition~$(A,B,C)$ of~$V(G)$ such that~$C$ is a clique with~$|C| = c$,~$C$ separates~$A$ and~$B$, and~$G_1,G_2$ are subgraphs induced by~$A\cup C$ and~$B\cup C$, respectively. A graph is a \emph{block graph} if it can be expressed recursively as the 1-clique sum of complete graphs. Block graphs can be equivalently defined using a distance metric. Given two vertices~$u,v$ of a graph, the \emph{distance}~$d(u,v)$ is the length of the shortest path between the vertices. A graph is a {block graph} if and only if it satisfies the \emph{four-point condition}: for any four vertices~$u,v,w,x$, the larger two of the following are equal:
\begin{align}\label{eqn:four point condition}
    d(u,v) + d(w,x),\quad d(u,w) + d(v,x),\quad d(u,x) + d(v,w).
\end{align}

\begin{theorem}[{\cite[Theorem 5]{misra2019gaussian}}]\label{thm:misra_sullivant}
    Let~$G$ be a block graph. Then, its vanishing ideal~$\I_G$ is generated by the~$2\times 2$ minors of~$\Sigma_{A\cup C,B\cup C}$, the submatrix of $\Sigma$ with rows indexed by $A\cup C$ and columns indexed by $B\cup C$), for all possible 1-clique partitions~$(A,B,C)$ of~$G$.
\end{theorem}

Adding colors and removing edges in a graph may increase or decrease the complexity of the model. For instance, any nontrivial coloring of a clique gives a nonzero ideal. However, when a block graph has an RCOP coloring, its model is still toric \cite{coons2023symmetrically}. In this case, it is precisely the uncolored model cut out by a linear space determined by symmetries in the coloring of the graph and it satisfies the conditions of \Cref{thm:ACTUAL_INTERSECTION} with~$F$ being the identity transformation.

\subsection{BMT-derived models}
\label{sec: combinatorics}

As shown in \Cref{fig:intro} and noted in \cite{sturmfels2020estimating}, we can associate a colored Gaussian graphical model to every BMT model such that the concentration matrices of the former coincide with the covariance matrices of the latter. As a consequence, CGG models derived from phylogenetic trees have toric vanishing ideals under the reduced Laplacian transformation. 
This observation is important to us, so we formalize it in the general setting of colored and zeroed trees.

Let~$\T=(T,\Lambda_\T,Z(\T))$ be a tree with colored and zeroed nodes on~$n$ non-root leaves.  The \emph{BMT-derived graph from~$\T$} is the  colored graph~$\G=(G,\Lambda_\G)$ on vertices~$\{1,\ldots, n\}$ such that
\begin{enumerate}
    \item vertices~$i,j$ have~$\lambda_\G(i)=\lambda_\G(j)$ in~$\G$  if and only if leaves~$i,j$  have $\lambda_\T(i)=\lambda_\T(j)$ in~$\T$, 
    \item~$\{i,j\}$ is an edge in~$\G$ if and only if~$\lca(i,j)\notin \rZ(\T)$,
    \item  edges~$\{i,j\}$,~$\{k,l\}$ in~$\G$ have~$\lambda_\G(\{i,j\})=\lambda_\G(\{k,l\})$ if and only if $\lca_\T(i,j)=\lca_\T(k,l)$.
\end{enumerate} 

\begin{proposition}
Let~$\T$ be a tree with colored and zeroed nodes, and let~$\G$ be its BMT-derived graph. Then,~$L_\T=\mathscr{L}_\G$. 
\end{proposition}

\begin{proof}
    This follows from the construction of BMT-derived graphs.
\end{proof}

This observation motivates the following definition. The \emph{BMT-derived model from~$\T$} is the colored Gaussian graphical model with linear space of concentration matrices~$L_\T=\L_\G$ intersected with the positive definite cone. 

For an uncolored tree~$T$ with~$Z(T)=\emptyset$, the associated BMT-derived graph is a colored complete graph~$\mathcal C$ where all vertices have different colors. In fact, all BMT-derived models can be obtained via a given set of graph operations on a complete graph~$\mathcal C$, as the following proposition shows.  

\begin{proposition}\label{prop:graph_from_colored_zeroed}
Let~$\mathcal{C}$ be the colored complete graph derived from an uncolored tree~$T$ with~$Z(T) = \emptyset$. The set of BMT-derived graphs obtained by adding colored and zeroed nodes to~$T$ is precisely the set of colored graphs obtained by any combination of the following operations on~$\mathcal{C}$:
\begin{multicols}{2} \begin{enumerate} \item merging vertex colors, \item deleting vertices, \item merging edge colors,  \item deleting edge colors. \end{enumerate} \end{multicols}
\end{proposition}

\begin{proof}
Setting two leaves~$i,j$ in the tree to share the same color corresponds to setting vertices~$i,j$ in the graph to share the same color. {Zeroing} a leaf~$\ell$ {in the tree} corresponds to deleting a vertex~$\ell$ in the graph. Setting two internal nodes~$k,m$ of the tree to have the same color corresponds to merging the edge colors of the graph that correspond to~$k$ and~$m$, {(all edges $\{i,j\}$ such that $\lca_T(i,j) = k$ or $m$).} {Zeroing} an internal node~$t$ corresponds to deleting the edges whose colors are given by~$t$. 
\end{proof}

\subsection{The generalized path map}  We naturally generalize 
the monomial map (\ref{eqn:shortest path map}) on~$T$ to the \emph{generalized path map} on a colored tree with zeroed nodes~$\T$ as follows: 
\begin{align} \label{eqn:generalized path map}
\varphi_\T: \C[p_{ij} \mid 0\leq i<j\leq n] \to \C[\theta_{\lambda} \mid \lambda\in \Lambda_\T], \ 
p_{ij} \mapsto \prod_{\substack{(\ell, k)\in i \leftrightsquigarrow j\\k\notin Z(\T)}} \theta_{\lambda_\T(k)}. 
\end{align} 
We will provide combinatorial conditions on~$\T$ and~$\G$ under which the reciprocal variety is toric under an appropriate linear transformation. In all these cases, the generalized path map~$\varphi_\T$, up to some small change, will successfully parametrize the model. 

\section{BMT-derived models from coloring the tree}\label{sec:4}

\subsection{Coloring leaves}
A colored graph is \emph{vertex-regular} if vertices of the same color are adjacent to edges of the same colors, counting multiplicity. First, we consider trees with color restrictions only on their leaves. In \Cref{prop:vertex_reg_iff}, we show that a BMT-derived graph from such a tree is vertex-regular if and only if leaves of the same color have the same parent. Here, we show that for any colored tree~$\T$ satisfying this, the reciprocal variety~$L_{\T}^{-1}$ is toric under the reduced graph Laplacian transformation. The following example shows that vertex-regularity is closely related to toric structure.

\begin{example}
In \Cref{fig:vertex_regular_tree}, we consider BMT-derived graphs obtained from different colorings of the leaves.
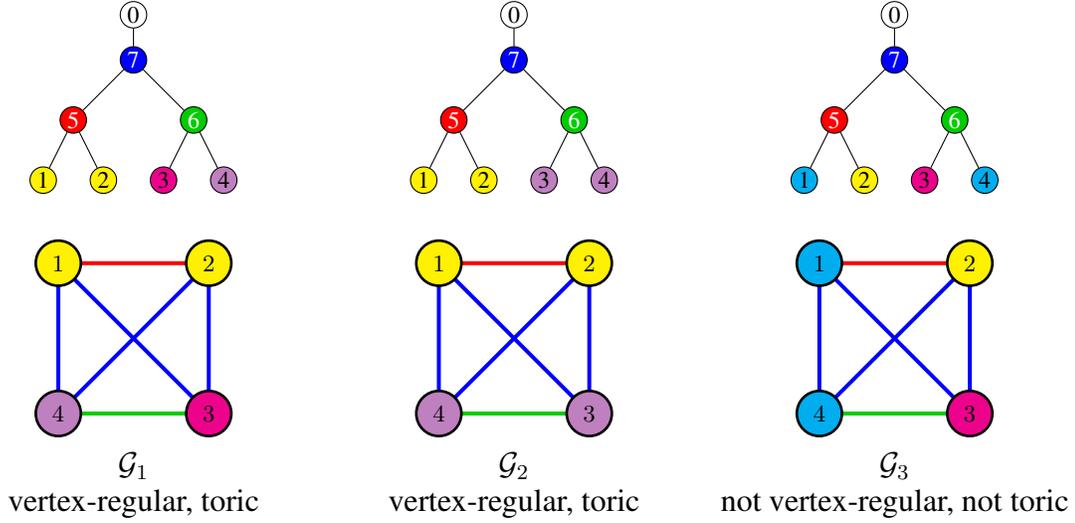
\begin{figure}[h]
\centering
\begin{subfigure}[b]{0.3\textwidth}
    \centering
    \begin{tikzpicture}[scale=0.4]
            \tikzset{
                VertexStyle/.style={
                    shape=circle,
                    draw,
                    fill=white,
                    minimum size=10pt,
                    inner sep=0pt,
                    font=\tiny
                },
                EdgeStyle/.style={
                    color=black,
                    line width=0.5pt,
                    draw=none, fill=none, midway
                }
            }
            \Vertex[x=0, y=5.5, label=0, color=gray!0, style={opacity=0}]{n0}
            \Vertex[x=-1, y=0, label=2, color=yellow]{n2}
            \Vertex[x=-3, y=0, label=1, color=yellow]{n1}
            \Vertex[x=1, y=0, label=3, color=magenta]{n3}
            \Vertex[x=3, y=0, label=4, color=violet!50]{n4}
            \Vertex[x=-2, y=2, label=\textcolor{white}{5}, color=red]{n5}
            \Vertex[x=2, y=2, label=\textcolor{white}{6}, color=darkgreen]{n6}
            \Vertex[x=0, y=4, label=\textcolor{white}{7}, color=blue]{n7}

            \draw (n0) -- (n7);
            \draw (n7) -- (n6);
            \draw (n6) -- (n4);
            \draw (n7) -- (n5);
            \draw (n5) -- (n2);
            \draw (n5) -- (n1);
            \draw (n6) -- (n3);
        \end{tikzpicture}
\end{subfigure}
\begin{subfigure}[b]{0.3\textwidth}
    \centering
    \begin{tikzpicture}[scale=0.4]
            \tikzset{
                VertexStyle/.style={
                    shape=circle,
                    draw,
                    fill=white,
                    minimum size=10pt,
                    inner sep=0pt,
                    font=\tiny
                },
                EdgeStyle/.style={
                    color=black,
                    line width=0.5pt,
                    draw=none, fill=none, midway
                }
            }
            \Vertex[x=0, y=5.5, label=0, color=gray!0, style={opacity=0}]{n0}
            \Vertex[x=-1, y=0, label=2, color=yellow]{n2}
            \Vertex[x=-3, y=0, label=1, color=yellow]{n1}
            \Vertex[x=1, y=0, label=3, color=violet!50]{n3}
            \Vertex[x=3, y=0, label=4, color=violet!50]{n4}
            \Vertex[x=-2, y=2, label=\textcolor{white}{5}, color=red]{n5}
            \Vertex[x=2, y=2, label=\textcolor{white}{6}, color=darkgreen]{n6}
            \Vertex[x=0, y=4, label=\textcolor{white}{7}, color=blue]{n7}

            \draw (n0) -- (n7);
            \draw (n7) -- (n6);
            \draw (n6) -- (n4);
            \draw (n7) -- (n5);
            \draw (n5) -- (n2);
            \draw (n5) -- (n1);
            \draw (n6) -- (n3);
        \end{tikzpicture}
\end{subfigure}
\begin{subfigure}[b]{0.3\textwidth}
    \centering
    \begin{tikzpicture}[scale=0.4]
            \tikzset{
                VertexStyle/.style={
                    shape=circle,
                    draw,
                    fill=white,
                    minimum size=10pt,
                    inner sep=0pt,
                    font=\tiny
                },
                EdgeStyle/.style={
                    color=black,
                    line width=0.5pt,
                    draw=none, fill=none, midway
                }
            }
            \Vertex[x=0, y=5.5, label=0, color=gray!0, style={opacity=0}]{n0}
            \Vertex[x=-1, y=0, label=2, color=yellow]{n2}
            \Vertex[x=-3, y=0, label=1, color=cyan]{n1}
            \Vertex[x=1, y=0, label=3, color=magenta]{n3}
            \Vertex[x=3, y=0, label=4, color=cyan]{n4}
            \Vertex[x=-2, y=2, label=\textcolor{white}{5}, color=red]{n5}
            \Vertex[x=2, y=2, label=\textcolor{white}{6}, color=darkgreen]{n6}
            \Vertex[x=0, y=4, label=\textcolor{white}{7}, color=blue]{n7}

            \draw (n0) -- (n7);
            \draw (n7) -- (n6);
            \draw (n6) -- (n4);
            \draw (n7) -- (n5);
            \draw (n5) -- (n2);
            \draw (n5) -- (n1);
            \draw (n6) -- (n3);
        \end{tikzpicture}
\end{subfigure}

\vspace{0.5cm}

\begin{subfigure}[t]{0.3\textwidth}
\centering
 \begin{tikzpicture}
    \Vertex[label=$1$,color=yellow]{A};
  \Vertex[x=2,label=$2$,color=yellow]{B};
{\Vertex[y=-2,label=$4$,color=violet!50]{D}};
    \Vertex[x=2,y=-2, label=$3$,color=magenta]{C};
    \Edge[color=red](A)(B);
    \Edge[color=blue](A)(D);
    \Edge[color=blue](B)(D);
    \Edge[color=darkgreen](C)(D);
    \Edge[color=blue](A)(C);
    \Edge[color=blue](B)(C);
\end{tikzpicture}
\begin{center}
{$\G_1$}

{vertex-regular, toric}
\end{center}
\end{subfigure}
    \begin{subfigure}[t]{0.3\textwidth}
\centering
\begin{tikzpicture}
    \Vertex[label=$1$,color=yellow]{A};
  \Vertex[x=2,label=$2$,color=yellow]{B};
{\Vertex[y=-2,label=$4$,color=violet!50]{D}};
    \Vertex[x=2,y=-2, label=$3$,color=violet!50]{C};
    \Edge[color=red](A)(B);
    \Edge[color=blue](A)(D);
    \Edge[color=blue](B)(D);
    \Edge[color=darkgreen](C)(D);
    \Edge[color=blue](A)(C);
    \Edge[color=blue](B)(C);
\end{tikzpicture}
\begin{center}
{$\G_2$}

{vertex-regular, toric}
\end{center}
\end{subfigure}
\begin{subfigure}[t]{0.3\textwidth}
        \centering 
        \begin{tikzpicture}
            \Vertex[label=$1$,color=cyan]{A};
        \Vertex[x=2,label=$2$,color=yellow]{B};
            \Vertex[y=-2,label=$4$,color=cyan]{D};
            \Vertex[x=2,y=-2, label=$3$,color=magenta]{C};
            \Edge[color=red](A)(B);
            \Edge[color=blue](A)(D);
            \Edge[color=blue](B)(D);
            \Edge[color=darkgreen](C)(D);
            \Edge[color=blue](A)(C);
            \Edge[color=blue](B)(C);
        \end{tikzpicture}
        \begin{center}
        {$\G_3$}

        {not vertex-regular, not toric}
        \end{center}
    \end{subfigure}
    \caption{Vertex-regular and non-vertex-regular colorings for the tree in \Cref{fig:uncolored}.}
     \label{fig:vertex_regular_tree}
\end{figure}
Under the reduced graph Laplacian transformation, the vanishing ideals of the reciprocal varieties are
    \vspace{-0.5em}
 \begin{align*}
  \I_{\G_1} &=  \langle p_{14} - p_{24},p_{13}-p_{23},p_{01}-p_{02},p_{04}p_{23}-p_{03}p_{24}\rangle\\
     &= I_{T} + \langle p_{14} - p_{24},p_{13}-p_{23},p_{01}-p_{02}\rangle\\
      \I_{\G_2} &= \langle p_{23}-p_{24},p_{14}-p_{24},p_{13}-p_{24},p_{03}-p_{04},p_{01}-p_{02}\rangle\\
      &= I_T + \langle p_{14} - p_{24},p_{13}-p_{23},p_{01}-p_{02},p_{13}-p_{14}, p_{23}-p_{24},p_{03}-p_{04}\rangle\\
 \I_{\G_3} &= \left\langle p_{14}p_{23}-p_{13}p_{24},\,p_{04}p_{23}-p_{03}p_{24},\,p_{02}p_{14}-p_{01}p_{24},\,p_{04}p_{13}-p_{03}p_{14},\right.\\
 &\quad\left.\,p_{02}p_{13}-p_{01}p_{23},\,p_{01}p_{02}p_{03}-p_{02}p_{03}p_{04}+p_{01}p_{03}p_{12} + \dots -p_{04}p_{24}p_{34}\right\rangle.
    \end{align*}
Vertex-regular coloring coincides with toricness in these coordinates. Additionally, the ideals corresponding to the vertex-regular graphs can be expressed as the sum of~$I_T$ with an ideal whose generators capture the vertex symmetries of the graph. We formalize this observation in \Cref{thm:vertex_regular}. Using the algorithm in~\cite{kahle2024lie}, we confirm that the vanishing ideal~$\I_{\G_3}$ of the non-vertex-regular graph is not toric under any linear change of variables. 
\end{example}
To facilitate the proofs, we discuss the combinatorics of vertex-regular colored graphs. The \emph{parent} of node~$i$, denoted~$\pa(i)$, is the internal node~$j$ such that there is a directed edge from~$j$ to~$i$.
\begin{proposition} \label{prop:vertex_reg_iff}
      Let~$\G$ be a BMT-derived complete graph from tree~$\T$. Then~$\G$ is vertex-regular if and only if leaves of the same color have the same parent. 
\end{proposition}

\begin{proof}
Suppose that the leaves~$i,j$ satisfy~$\lambda_\T(i) = \lambda_\T(j)$ and that leaf~$j$ is not a descendant of~$\pa(i)$. Then vertex~$j$ is not adjacent to any edge with color~$\lambda_\T(\pa(i))$, so leaf~$j$ must be a descendant of~$\pa(i)$. By symmetry, leaf~$i$ is also a descendant of~$\pa(j)$, so the result follows.
\end{proof}
An immediate corollary is the following.
\begin{corollary}
    Let~$\G$ be a BMT-derived complete graph from tree~$\T$. If $\lambda_{\G}(i)=\lambda_{\G}(j)$ for some pair of vertices~$i, j$, then~$\lambda_{\G}(\{i,k\})=\lambda_{\G}(\{j,k\})$ for any vertex~$k\neq i,j$. 
\end{corollary}

Given a vertex-regular BMT-derived graph~$\G$, we define its {\textit{vertex-regular completion}}~$\overline{\G}$ as the complete colored graph on the same vertex set as~$\G$ whose coloring captures the vertex symmetries of~$\G$ as follows: 
\begin{enumerate}
    \item    if~$\lambda_{{\G}}(i) =\lambda_{{\G}}(j)$ in~$\G$, then~$\lambda_{\overline{\G}}(i) =\lambda_{\overline{\G}}(j)$ in~$\overline{\G}$,
    \item  if~$\lambda_{{\G}}(i) = \lambda_{{\G}}(j)$ in~$\G$, then~$\lambda_{\overline{\G}}(\{i,k\}) =\lambda_{\overline{\G}}(\{j,k\})$ for  all~$k\neq i,j$ in~$\overline{\G}$,
    \item all remaining vertices and edges have other distinct colors. 
\end{enumerate}

The linear space~$\L_{\overline{\G}}$ contains~$\L_\G$, so~$\L^{-1}_\G\subseteq \L^{-1}_{\overline{\G}}$. The vanishing ideal of~$\L_{\overline{\G}}^{-1}$ is determined by binomial linear conditions, which includes all linear relations in the vanishing ideal of~$\L^{-1}_\G$. To prove this, it is sufficient to show that~$\L_{\overline{\G}}^{-1}$ is a \emph{Jordan algebra}. Recall that a linear space of symmetric matrices in~$\sym_n(\R)$ is a Jordan algebra if and only if it is closed under the~$\bullet$ operation, where~$\bullet$ on~$\sym_n(\R)$ is defined as~$$X\bullet Y = \frac{1}{2}(XY + YX).$$
We use the following result from Jensen to prove \Cref{lemma:jensen_corollary}.
\begin{lemma}[{\cite[Lemma 1]{jensen1988covariance}}]\label{jensen88} 
Let~$\L$ be a linear space of symmetric matrices and~$\L^{-1}$ its reciprocal variety. Then,~$\L = \L^{-1}$ if and only if~$\L$ is a subalgebra of the Jordan algebra on~$\sym_n(\R)$.  
\end{lemma}

\begin{lemma}\label{lemma:jensen_corollary} Let~$\G$ be a BMT-derived graph and let~$\overline{\G}$ be its vertex-regular completion. Then
$\L_{\overline{\G}}^{-1} = \L_{\overline{\G}}$ and it has vanishing toric linear ideal 
\begin{align*}\label{def_JGprime_ideal}
\left\langle \sigma_{ii} - \sigma_{jj}, \sigma_{ik} - \sigma_{jk} \mid  i,j \in [n], \lambda_{\overline{\G}}(i) = \lambda_{\overline{\G}}(j), k\in  [n]\setminus \{i,j\}\right\rangle.
\end{align*}
\end{lemma}

\begin{proof}
    By Lemma \ref{jensen88}, it suffices to show that~$\mathscr{L}_{\overline{\G}}$ is a Jordan algebra. We can view~$\L_{\overline{\G}}$ as the intersection of linear spaces~$\L_{ij}$, where~$\L_{ij}$ is defined as follows for any pair of vertices~$i,j$ sharing the same color in~$\G$:
    \[
    \L_{ij} = \left\{ T = (t_{\ell m})_{1\leq  \ell \leq  m \leq n} \in \sym_n(\R) \mid t_{ii} = t_{jj}, \, t_{ik} = t_{jk}  \,\forall \,k\not=i,j\right\}.
    \]
    Let~$A,B\in \L_{ij}$. It suffices to show that~$C = A\bullet B \in \L_{ij}$. It is easy to check that~$C_{ii} = C_{jj}$. Given~$k\neq i,j$, we have
\begin{align*}
  2(C_{ik} - C_{jk}) &=  \sum_{\ell=1}^n a_{i\ell} b_{\ell k} +  b_{i\ell} a_{\ell k} - \sum_{\ell=1}^n a_{j\ell} b_{\ell k} +  b_{j\ell} a_{\ell k} = 0,
\end{align*}
because~$a_{ii} = a_{jj},b_{ii} = b_{jj}$ and~$a_{i\ell} = a_{j\ell}, b_{i\ell} = b_{j\ell}$ for~$\ell\not=i,j$. Therefore,~$C_{ik} = C_{jk}$ for all~$k \neq i,j$, so~$C\in \L_{ij}$. Then~$\L_{ij}$ is a Jordan algebra. As such,~$\L_{\overline{\G}}$ is the intersection of Jordan algebras and thus it is also a Jordan algebra. {The generators of the vanishing ideal of $\mathcal{L}_{ij}$ are $t_{ii} - t_{jj}, t_{ik} - t_{jk}$ for all $k\not=i,j$, so the vanishing ideal of $\mathcal{L}_{\mathcal{\overline{G}}} = \bigcap_{\lambda_{\overline{\mathcal{G}}}(i) = \lambda_{\overline{\mathcal{G}}}(j)} \mathcal{L}_{ij}$ is 
\begin{align*}
    \langle t_{ii}-t_{jj}, t_{ik}-t_{ij}\;|\; i,j\in [n], \lambda_{\overline{\mathcal{G}}}(i) = \lambda_{\overline{\mathcal{G}}}(j), k\in [n]\setminus\{i,j\}\rangle. 
\end{align*}}\end{proof}
Before we state the main result of the section, note that the \emph{generalized path map}  (\ref{eqn:generalized path map}) on the colored tree~$\T$ with no zeroed nodes has the form: 
\begin{align} \label{eqn:generalized path map vertex}
\varphi_\T: \C[p_{ij} \mid 0\leq i<j\leq n] \to \C[\theta_{\lambda} \mid \lambda\in \Lambda_\T], \ 
p_{ij} \mapsto \prod_{\substack{(\ell, k)\in i \leftrightsquigarrow j}} \theta_{\lambda_\T(k)}.
\end{align}
Let~$\I_{\overline{\G}}$ be the ideal in~$\C[p_{ij} \mid  0\leq i < j \leq n]$ obtained by substitution of variables in the ideal given in \Cref{lemma:jensen_corollary} as follows:
\begin{align*}
    \sigma_{ii} = p_{0i}\textrm{, and }\sigma_{ij} = p_{ij} \textrm{ for }1\leq i< j\leq n.
\end{align*}
\begin{theorem}\label{thm:vertex_regular}
Let~$\T = (T,\lambda_\T,\emptyset)$ be a colored tree,
let~${\G}$ be its BMT-derived graph, and let {${\overline{\G}}$ be the vertex-regular completion of~$\G$}. Then, $\L_{\G}^{-1}=L_T^{-1} \cap \L_{\overline{\G}}^{-1}$ is toric under the reduced graph Laplacian transformation. It has toric vanishing ideal~$\sqrt{I_T+ \I_{\overline{\G}}}=\ker \varphi_\T$, where~$\varphi_\T$ is the generalized path map for~$\T$ given in (\ref{eqn:generalized path map vertex}).
\end{theorem}
\begin{proof}
First, observe that the vanishing ideal for~$\L_{\overline{\G}}^{-1}$ given in \Cref{lemma:jensen_corollary}  under the reduced graph Laplacian transformation is precisely the toric ideal~$\I_{\overline{\G}}$. Indeed, 
when~$\lambda_\T(i)=\lambda_\T(j)$ we get the following mapping of binomials:
\begin{align*}
\sigma_{ik}-\sigma_{jk}&\mapsto p_{ik}-p_{jk} \qquad \text{ when } k\neq i,j,\\
\sigma_{ii}-\sigma_{jj}&\mapsto p_{0i}-p_{0j} +\left(\sum_{ k\neq i,j}p_{ik}-p_{jk}\right).
\end{align*}
We can use the first set of quadrics above to substitute the last generator with~$p_{0i}-p_{0j}$ and get a binomial set of generators. The ideal~$I_T$ {in the reduced graph Laplacian coordinates} is toric by \Cref{thm:brownian_toric}.  

The irreducible variety~$L_\T$ is the intersection of irreducible varieties~$L_T$ and~$\L_{\overline \G}$. Hence, by \Cref{thm:ACTUAL_INTERSECTION},~$\L_{\G}^{-1}=L_T^{-1} \cap \L_{\overline{\G}}^{-1}$ has vanishing ideal $\sqrt{I_T+\I_{\overline {\G}}}$, which is toric under the reduced graph Laplacian transformation.

Lastly, we show that~$\sqrt{I_T+\I_{\overline \G}}=\ker \varphi_\T$.  Observe that any generator of~$I_T + \I_{\overline{\G}}$ is mapped to zero under~$\varphi_\T$, so~$I_T + \I_{\overline{\G}}\subseteq \ker \varphi_\T$. Taking radicals on both sides, we get~$\sqrt{I_T + \I_{\overline{\G}}}\subseteq \ker \varphi_\T$, because~$\ker \varphi_\T$ is radical. Note that~$\dim(\ker \varphi_\T)=\dim\left(\sqrt{I_T + \I_{\overline{\G}}}\right)=|\Lambda_\T|$.
Since a prime ideal cannot strictly contain another prime ideal of the same dimension, we obtain~$\sqrt{I_T + \I_{\overline{\G}}}= \ker \varphi_\T$.
\end{proof}

For all non-vertex-regular BMT-derived graphs on four vertices and all other non-vertex-regular BMT-derived graphs that we tested, we found that the vanishing ideal is not toric under any linear change of variables.  We conjecture that this is true for all complete BMT-derived graphs.

\begin{conjecture}
Let~$\T$ have distinct colors on all internal nodes. Then,~$L_\T^{-1}$ is toric if and only if leaves with the same color have the same parent. Equivalently, $L_\T^{-1}$ is toric if and only if its complete BMT-derived graph~$\G$ is vertex-regular. 
\end{conjecture}

\subsection{Coloring internal nodes} 

  Internal nodes~$i,j\in \Int(T)$ are \emph{adjacent} if~$\pa(i) = j$ or~$\pa(j) = i$.
 We show that merging colors of adjacent internal nodes induces a toric reciprocal variety. 
\begin{example}
In \Cref{fig:merge_internal}, the colored tree has adjacent internal nodes~$6$ and~$7$ of the same color. We see that the BMT-derived graph is also induced by the uncolored tree on the right.

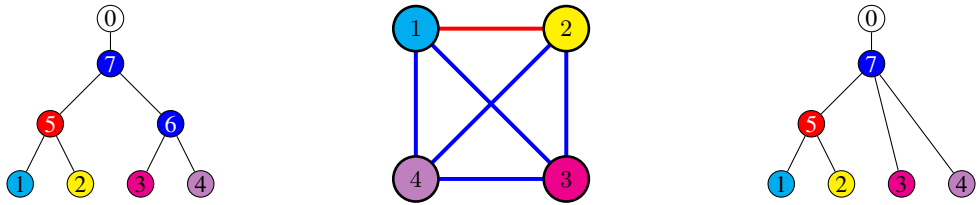
\begin{figure}[h]
    \centering
    \begin{subfigure}[b]{0.3\textwidth}
        \centering
        \begin{tikzpicture}[scale=0.4]
            \tikzset{
                VertexStyle/.style={
                    shape=circle,
                    draw,
                    fill=white,
                    minimum size=10pt,
                    inner sep=0pt,
                    font=\tiny
                },
                EdgeStyle/.style={
                    color=black,
                    line width=0.5pt,
                    draw=none, fill=none, midway
                }
            }
            \Vertex[x=0, y=5.5, label=0, color=gray!0, style={opacity=0}]{n0}
            \Vertex[x=-1, y=0, label=2, color=yellow]{n2}
            \Vertex[x=-3, y=0, label=1, color=cyan]{n1}
            \Vertex[x=1, y=0, label=3, color=magenta]{n3}
            \Vertex[x=3, y=0, label=4, color=violet!50]{n4}
            \Vertex[x=-2, y=2, label=\textcolor{white}{5}, color=red]{n5}
            \Vertex[x=2, y=2, label=\textcolor{white}{6}, color=blue]{n6}
            \Vertex[x=0, y=4, label=\textcolor{white}{7}, color=blue]{n7}

            \draw (n0) -- (n7);
            \draw (n7) -- (n6);
            \draw (n6) -- (n4);
            \draw (n7) -- (n5);
            \draw (n5) -- (n2);
            \draw (n5) -- (n1);
            \draw (n6) -- (n3);
        \end{tikzpicture}
    \end{subfigure}
    \begin{subfigure}[b]{0.3\textwidth}
        \centering
        \begin{tikzpicture}
        \Vertex[label=$1$,color=cyan]{A};
      \Vertex[x=2,label=$2$,color=yellow]{B};
      {\Vertex[y=-2,label=$4$,color=violet!50]{D}};
        \Vertex[x=2,y=-2, label=$3$,color=magenta]{C};
        \Edge[color=red](A)(B);
        \Edge[color=blue](A)(D);
        \Edge[color=blue](B)(D);
        \Edge[color=blue](C)(D);
        \Edge[color=blue](A)(C);
        \Edge[color=blue](B)(C);
    \end{tikzpicture}
    \end{subfigure}
    \begin{subfigure}[b]{0.3\textwidth}
        \centering
        \begin{tikzpicture}[scale=0.4]
            \tikzset{
                VertexStyle/.style={
                    shape=circle,
                    draw,
                    fill=white,
                    minimum size=10pt,
                    inner sep=0pt,
                    font=\tiny
                },
                EdgeStyle/.style={
                    color=black,
                    line width=0.5pt,
                    draw=none, fill=none, midway
                }
            }
            \Vertex[x=0, y=5.5, label=0, color=gray!0, style={opacity=0}]{n0}
            \Vertex[x=-1, y=0, label=2, color=yellow]{n2}
            \Vertex[x=-3, y=0, label=1, color=cyan]{n1}
            \Vertex[x=1, y=0, label=3, color=magenta]{n3}
            \Vertex[x=3, y=0, label=4, color=violet!50]{n4}
            \Vertex[x=-2, y=2, label=\textcolor{white}{5}, color=red]{n5}
            \Vertex[x=0, y=4, label=\textcolor{white}{7}, color=blue]{n7}

            \draw (n0) -- (n7);
            \draw (n7) -- (n4);
            \draw (n7) -- (n5);
            \draw (n5) -- (n2);
            \draw (n5) -- (n1);
            \draw (n7) -- (n3);
        \end{tikzpicture}
    \end{subfigure}
    \caption{The tree with adjacent internal nodes 6 and 7 both colored blue induces the same graph as a tree where all internal nodes have distinct colors (uncolored).}
    \label{fig:merge_internal}
\end{figure}

Under the reduced graph Laplacian transformation, the model has toric vanishing ideal:
\begin{align*}
   & \left\langle p_{03}p_{24}\right. - p_{02}p_{34},\ p_{14}p_{23} - p_{13}p_{24},\ p_{04}p_{23} - p_{02}p_{34},\ p_{03}p_{14} - p_{01}p_{34},\\
   &\left. \ p_{02}p_{14} - p_{01}p_{24},\ p_{04}p_{13} - p_{01}p_{34},\ p_{02}p_{13} - p_{01}p_{23} \right\rangle.
\end{align*}
\end{example}

\begin{proposition}
Let~$\T = (T, \Lambda_{\T},\emptyset)$ be a colored tree such that only adjacent internal nodes share colors. Then, the vanishing ideal~$I_{\T}$ is toric under the reduced graph Laplacian transformation.
\end{proposition}

\begin{proof}

We construct an uncolored tree~$T'$ such that~$I_\T = I_{T'}$, as in \Cref{fig:merge_internal}. Let~$T'$ be a tree on~$n$ non-root leaves, with an internal node corresponding to each distinct color of internal nodes in~$\T$. Given an internal node~$i$ of~$T'$ and any~$j\in \Lv(T')\cup\Int(T')$, we let~$\pa(j) = i$ in~$T'$ if~$\lambda_\T(\pa(j)) = \lambda_\T(i)$. Since~$I_{\T} = I_{T'}$, the stament follows from \Cref{thm:brownian_toric}.
\end{proof}

\begin{example}
In \Cref{fig:non-adj}, the tree has non-adjacent internal nodes 5 and 6 of the same color.

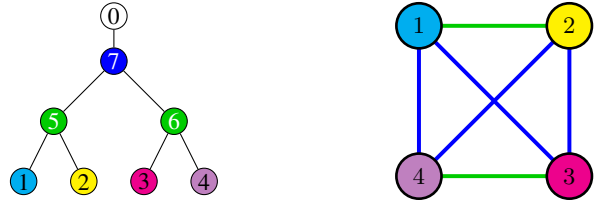
\begin{figure}[h]
    \centering
   \begin{subfigure}[b]{0.3\textwidth}
        \centering
        \begin{tikzpicture}[scale=0.4]
            \tikzset{
                VertexStyle/.style={
                    shape=circle,
                    draw,
                    fill=white,
                    minimum size=10pt,
                    inner sep=0pt,
                    font=\tiny
                },
                EdgeStyle/.style={
                    color=black,
                    line width=0.5pt,
                    draw=none, fill=none, midway
                }
            }
            \Vertex[x=0, y=5.5, label=0, color=gray!0, style={opacity=0}]{n0}
            \Vertex[x=-1, y=0, label=2, color=yellow]{n2}
            \Vertex[x=-3, y=0, label=1, color=cyan]{n1}
            \Vertex[x=1, y=0, label=3, color=magenta]{n3}
            \Vertex[x=3, y=0, label=4, color=violet!50]{n4}
            \Vertex[x=-2, y=2, label=\textcolor{white}{5}, color=darkgreen]{n5}
            \Vertex[x=2, y=2, label=\textcolor{white}{6}, color=darkgreen]{n6}
            \Vertex[x=0, y=4, label=\textcolor{white}{7}, color=blue]{n7}

            \draw (n0) -- (n7);
            \draw (n7) -- (n6);
            \draw (n6) -- (n4);
            \draw (n7) -- (n5);
            \draw (n5) -- (n2);
            \draw (n5) -- (n1);
            \draw (n6) -- (n3);
        \end{tikzpicture}
    \end{subfigure}
    \begin{subfigure}[b]{0.3\textwidth}
        \centering
        \begin{tikzpicture}
        \Vertex[label=$1$,color=cyan]{A};
  \Vertex[x=2,label=$2$,color=yellow]{B};
  {\Vertex[y=-2,label=$4$,color=violet!50]{D}};
    \Vertex[x=2,y=-2, label=$3$,color=magenta]{C};
    \Edge[color=darkgreen](A)(B);
    \Edge[color=blue](A)(D);
    \Edge[color=blue](B)(D);
    \Edge[color=darkgreen](C)(D);
    \Edge[color=blue](A)(C);
    \Edge[color=blue](B)(C);
    \end{tikzpicture}
    \end{subfigure}
\caption{Tree with non-adjacent internal nodes sharing the same color.}
\label{fig:non-adj}
\end{figure}
Under the reduced graph Laplacian transformation, the vanishing ideal of the model is not toric:
\[
\resizebox{0.99\linewidth}{!}{
    $\langle p_{14}p_{23} - p_{13}p_{14}, \ldots, - p_{02} p_{12} p_{34} + p_{03} p_{12} p_{34} + p_{04} p_{12} p_{34} - 2p_{01} p_{23} p_{34} - 2p_{01} p_{24} p_{34}\rangle.$}
\]
Checking with \cite{kahle2024lie}, we find that this ideal is not toric under any other linear change of variables. 
\end{example}
All the examples we computed for merging non-adjacent internal nodes lead to ideals that are not toric under any linear change of variables. This provides evidence for the following conjecture. 

\begin{conjecture} Vanishing ideals of BMT-derived models constructed by merging color classes of edges that correspond to non-adjacent internal nodes are not toric under any invertible linear change of variables. 
\end{conjecture}

\section{Zeroed nodes that imply toric structure}\label{sec:5}
Here, we impose symmetries by {zeroing out parameters for nodes of trees} with no additional colorings imposed on the nodes or leaves. Since zeroing leaves is equivalent to considering a subtree of the original model, we focus on zeroing only internal nodes. By \Cref{prop:graph_from_colored_zeroed}, this corresponds to deleting edges of BMT graphs. 
We show that a tree with zeroed nodes whose BMT derived graph~$\G$ is a block graph has vanishing ideal that is toric under the~$G$-graph derived Laplacian described in (\ref{eqn:derived_laplacian}). 

\begin{example}\label{ex:block_nonblock}
 In Figure \ref{fig:block_deleted}, we consider graphs arising from zeroing nodes in trees.

 \begin{figure}[h]
 \centering
\begin{subfigure}[b]{0.3\textwidth}
    \centering
    \begin{tikzpicture}[scale=0.4]
            \tikzset{
                VertexStyle/.style={
                    shape=circle,
                    draw,
                    fill=white,
                    minimum size=10pt,
                    inner sep=0pt,
                    font=\tiny
                },
                EdgeStyle/.style={
                    color=black,
                    line width=0.5pt,
                    draw=none, fill=none, midway
                }
            }
            \Vertex[x=0, y=5.5, label=0, color=gray!0, style={opacity=0}]{n0}
            \Vertex[x=-1, y=0, label=2, color=yellow]{n2}
            \Vertex[x=-3, y=0, label=1, color=cyan]{n1}
            \Vertex[x=1, y=0, label=3, color=magenta]{n3}
            \Vertex[x=3, y=0, label=4, color=violet!50]{n4}
            \Vertex[x=-2, y=1.33, label=\textcolor{white}{5}, color=red]{n5}
            \Vertex[x=-1, y=2.67, label=\textcolor{white}{6}, color=darkgreen]{n6}
            \Vertex[x=0, y=4, label=\textcolor{white}{7}, color=blue]{n7}

            \draw (n0) -- (n7);
            \draw (n7) -- (n6);
            \draw (n7) -- (n4);
            \draw (n6) -- (n5);
            \draw (n5) -- (n2);
            \draw (n5) -- (n1);
            \draw (n6) -- (n3);
        \end{tikzpicture}
\end{subfigure}
\begin{subfigure}[b]{0.3\textwidth}
    \centering
    \begin{tikzpicture}[scale=0.4]
            \tikzset{
                VertexStyle/.style={
                    shape=circle,
                    draw,
                    fill=white,
                    minimum size=10pt,
                    inner sep=0pt,
                    font=\tiny
                },
                EdgeStyle/.style={
                    color=black,
                    line width=0.5pt,
                    draw=none, fill=none, midway
                }
            }
            \Vertex[x=0, y=5.5, label=0, color=gray!0, style={opacity=0}]{n0}
            \Vertex[x=-1, y=0, label=2, color=yellow]{n2}
            \Vertex[x=-3, y=0, label=1, color=cyan]{n1}
            \Vertex[x=1, y=0, label=3, color=magenta]{n3}
            \Vertex[x=3, y=0, label=4, color=violet!50]{n4}
            \Vertex[x=-2, y=1.33, label=\textcolor{white}{5}, color=red]{n5}
            \Vertex[x=-1, y=2.67, label={6}, color=white]{n6}
            \Vertex[x=0, y=4, label=\textcolor{white}{7}, color=blue]{n7}

            \draw (n0) -- (n7);
            \draw (n7) -- (n6);
            \draw (n7) -- (n4);
            \draw (n6) -- (n5);
            \draw (n5) -- (n2);
            \draw (n5) -- (n1);
            \draw (n6) -- (n3);
        \end{tikzpicture}
\end{subfigure}
\begin{subfigure}[b]{0.3\textwidth}
    \centering
    \begin{tikzpicture}[scale=0.4]
            \tikzset{
                VertexStyle/.style={
                    shape=circle,
                    draw,
                    fill=white,
                    minimum size=10pt,
                    inner sep=0pt,
                    font=\tiny
                },
                EdgeStyle/.style={
                    color=black,
                    line width=0.5pt,
                    draw=none, fill=none, midway
                }
            }
            \Vertex[x=0, y=5.5, label=0, color=gray!0, style={opacity=0}]{n0}
            \Vertex[x=-1, y=0, label=2, color=yellow]{n2}
            \Vertex[x=-3, y=0, label=1, color=cyan]{n1}
            \Vertex[x=1, y=0, label=3, color=magenta]{n3}
            \Vertex[x=3, y=0, label=4, color=violet!50]{n4}
            \Vertex[x=-2, y=1.33, label={5}, color=white]{n5}
            \Vertex[x=-1, y=2.67, label=\textcolor{white}{6}, color=darkgreen]{n6}
            \Vertex[x=0, y=4, label=\textcolor{white}{7}, color=blue]{n7}

            \draw (n0) -- (n7);
            \draw (n7) -- (n6);
            \draw (n7) -- (n4);
            \draw (n6) -- (n5);
            \draw (n5) -- (n2);
            \draw (n5) -- (n1);
            \draw (n6) -- (n3);
        \end{tikzpicture}
\end{subfigure}

\vspace{0.5cm}

\begin{subfigure}{0.3\textwidth}
\centering
 \begin{tikzpicture}
    \Vertex[label=$1$,color=cyan]{A};
  \Vertex[x=2,label=$2$,color=yellow]{B};
{\Vertex[y=-2,label=$4$,color=violet!50]{D}};
    \Vertex[x=2,y=-2, label=$3$,color=magenta]{C};
    \Edge[color=red](A)(B);
    \Edge[color=blue](A)(D);
    \Edge[color=blue](B)(D);
    \Edge[color=blue](C)(D);
    \Edge[color=darkgreen](A)(C);
    \Edge[color=darkgreen](B)(C);
\end{tikzpicture}
\begin{center}
{$\G_1$}

{block, toric}

\end{center}
\end{subfigure}
    \begin{subfigure}{0.3\textwidth}
\centering
\begin{tikzpicture}
    \Vertex[label=$1$,color=cyan]{A};
  \Vertex[x=2,label=$2$,color=yellow]{B};
{\Vertex[y=-2,label=$4$,color=violet!50]{D}};
    \Vertex[x=2,y=-2, label=$3$,color=magenta]{C};
    \Edge[color=red](A)(B);
    \Edge[color=blue](A)(D);
    \Edge[color=blue](B)(D);
    \Edge[color=blue](C)(D);
\end{tikzpicture}
\begin{center}
{$\G_2$}

{block, toric}
\end{center}
\end{subfigure}
\begin{subfigure}{0.3\textwidth}
        \centering
        \begin{tikzpicture}
        \Vertex[label=$1$,color=cyan]{A};
      \Vertex[x=2,label=$2$,color=yellow]{B};
    {\Vertex[y=-2,label=$4$,color=violet!50]{D}};
        \Vertex[x=2,y=-2, label=$3$,color=magenta]{C};
        \Edge[color=blue](A)(D);
        \Edge[color=blue](B)(D);
        \Edge[color=blue](C)(D);
        \Edge[color=darkgreen](A)(C);
        \Edge[color=darkgreen](B)(C);
    \end{tikzpicture}
        \begin{center}
        {$\G_3$}

        {not block, not toric}
        \end{center}
    \end{subfigure}

    \caption{Trees with zeroed nodes {(in white)} and their corresponding BMT-derived graphs
    .}
     \label{fig:block_deleted}
\end{figure}
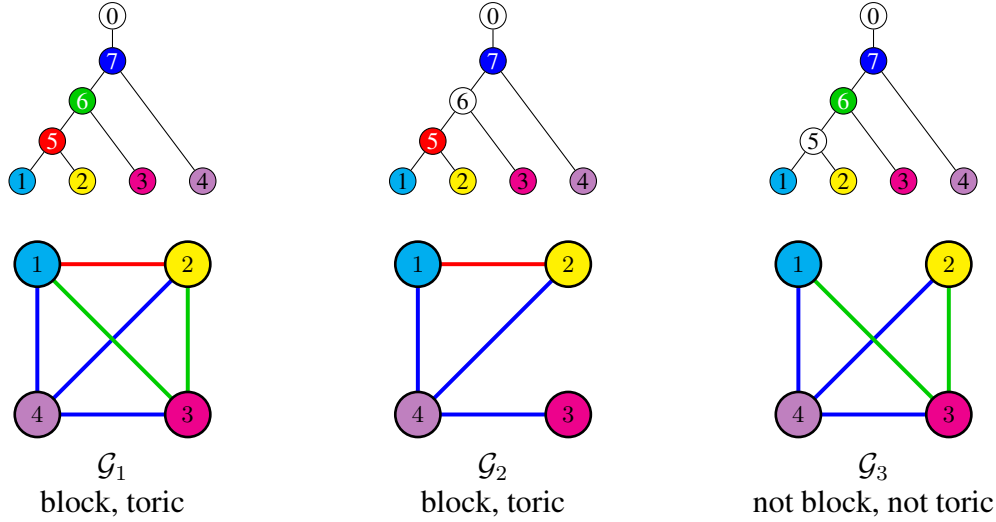

Here are their vanishing ideals, where we take~$\I_{\G_1}$ in coordinates (\ref{eqn:graph_laplacian_BMT}) and~$\I_{\G_2}$ in coordinates (\ref{eqn:derived_laplacian}):
\[
\begin{array}{ccl}
     \I_{\G_1} &= &(p_{03}p_{24}-p_{02}p_{34}, p_{14}p_{23}-p_{13}p_{24},p_{03}p_{14}-p_{01}p_{34}, p_{02}p_{14}-p_{01}p_{24},\\
     & &p_{02}p_{13}-p_{01}p_{23})\\
     \I_{\G_2}&= &(p_{03}p_{24}-p_{02}p_{34},p_{14}p_{23}-p_{13}p_{24},p_{04}p_{23}-p_{24}p_{34},p_{03}p_{14}-p_{01}p_{34},\\
     &&p_{02}p_{14}-p_{01}p_{24},
p_{04}p_{13}-p_{14}p_{34},p_{02}p_{13}-p_{01}p_{23}).
\end{array}
\]
Using code from~\cite{kahle2024lie}, we find that there is no linear change of variables under which~$\I_{\G_3}$ is toric.
\end{example}
\subsection{Combinatorics of trees with zeroed nodes}
For the rest of this paper, we assume the setup of a BMT graph~$\CC$ arising from a tree~$T$ and a BMT-derived block graph~$\G$ constructed from removing edges of same color class from~$\CC$.  

Lemmas \ref{lemma:height_geq_2}, \ref{lemma:leaves_geq_2}, and \ref{lemma:1_leaf} and \Cref{cor:ht top node} give  combinatorial description of tree~$\T$ with~$Z(\T)\neq \emptyset$ with BMT derived block graph, and ultimately help to characterize these BMT derived graphs in \Cref{prop:star_structure}. We need the following notation.   The \emph{height},~$\height(i)$, of an internal node~$i$ of a rooted tree~$T$ is the number of edges in the shortest path from the node to any non-root leaf of~$T$. We say~$j\in \Lv(T)\cup \Int(T)$ is a \emph{descendant} of~$i\in \Lv(T)\cup\Int(T)$ if there is a directed path from~$i$ to~$j$, and we let~$\desc(i)$ denote the set of all internal nodes descendant from~$i$. 

\begin{lemma}\label{lemma:height_geq_2}
 Let~$\G$ be a BMT-derived block graph from tree $\T = (T,\Lambda_{\T}, Z(\T))$. Let~$i \in \Int(T)\setminus Z(\T)$ with~$\height(i)\geq 2$. Then, we have~$j\in \Int(T)\setminus \mathrm{Z}(\T)$ for all internal nodes~$j\in \desc(i)$. 
\end{lemma}

\begin{proof}
Since~$\height(i)\geq 2$, there exists at least two branches (non-leaf) descendant from~$i$. Pick distinct leaves~$u,v$ from one branch and~$w,x$ from another branch. Then, vertices~$u,v,w,x$ of~$\G$ are connected by edges $\{u,w\}, \{u,x\}$, $\{v,w\}, \{v,x\}$ of $\G$ because~$i\not\in Z(\T)$. The four-point condition (\ref{eqn:four point condition}) implies that the subgraph on~$u,v,w,x$ is fully connected. Hence, we have $\lca(u,v), \lca(w,x)\in \Int(T)\setminus \mathrm{Z}(\T)$. By considering all choices of~$u,v,w,x$, the result follows. 
\end{proof}
To ensure the BMT-derived graph is connected, we keep the top internal node of~$\T$. If the graph is disconnected, our results hold for each of the disjoint subgraphs.  With this, \Cref{lemma:height_geq_2} implies that the top internal node must have a leaf as descendant.
\begin{corollary} \label{cor:ht top node}
    Let~$\T$ be a colored and zeroed tree with~$\mathrm{Z}(\T)\not=\emptyset$ that induces block graph~$\G$. Then, the top internal node~$i$ of~$T$ satisfies~$\height(i) = 1$.
\end{corollary}

\begin{lemma}\label{lemma:leaves_geq_2} 
Let~$\G$ be a BMT-derived block graph from tree~$\T = (T,\Lambda_\T,\mathrm{Z}(\T))$, with~$\mathrm{Z}(\T)\not=\emptyset$. Let~$i\not\in \mathrm{Z}(\T)$ such that there are at least two leaves with parent~$i$. Then, we have~$ \desc(i)\subset \Int(T)\setminus \mathrm{Z}(\T)$. 
\end{lemma}
\begin{proof}
Let~$c,c'$ be two leaves with parent~$i$. For internal node $j\in \desc(i)$, let $w,x$ be leaves satisfying $\lca(w,x) = j$. Since $i\not\in \mathrm{Z}(\T)$, we find \\$\{c,c'\},\{c,w\},\{c,x\},\{c',w\},\{c',x\}$ are edges in~$\G$. The four-point condition (\ref{eqn:four point condition}) on~$c,c',w,x$ implies~$\{w,x\}\in E(\G)$. Hence,~$j\notin \mathrm{Z}(\T)$. \end{proof}

\begin{lemma}\label{lemma:1_leaf}
   Let~$\G$ be a BMT-derived block graph from tree~$\T = (T,\Lambda_\T,\mathrm{Z}(\T))$, with~$\mathrm{Z}(\T)\not=\emptyset$. Let~$i\not\in \mathrm{Z}(\T)$ such that there is exactly one leaf with parent~$i$. Suppose~$j\in \desc(i)$ and~$j\not\in \mathrm{Z}(\T)$. Then, we have~$\desc(j)\subset \Int(T)\setminus \mathrm{Z}(\T)$.
\end{lemma}

\begin{proof}
{For any internal node~$k\in \desc(j)$, let~$w,x$ satisfy~$\lca(w,x) = k$. Pick leaf~$v$ such that~$\lca(w,v) = j$. Let~$c$ be the leaf with parent~$i$. Since~$i,j\in \Int(\T)\setminus Z(\T)$, we see that~$\{c,v\}$,~$\{c,w\}$,~$\{c,x\}$,~$\{v,w\}$, and~$\{v,x\}$ are edges in~$\G$. By the four-point condition (\ref{eqn:four point condition}), we have~$\{w,x\}\in E(\G)$. So~$k=\lca(w,x)\in \Int(T)\setminus Z(\T)$. }\end{proof}

 We now use the above lemmas to show that all BMT-derived block graphs are \textit{star graphs}, i.e. union of cliques intersecting at a unique vertex, referred to as the \emph{central} vertex.  

\begin{proposition}\label{prop:star_structure}
   Every BMT-derived block graph is a star graph.
\end{proposition}
\begin{proof}
{Let~$\G$ be a non-complete BMT-derived block graph from tree~$\T$. By Lemmas \ref{lemma:height_geq_2}, \ref{lemma:leaves_geq_2}, and \ref{lemma:1_leaf}, the top internal node~$i$ of~$\T$ satisfies~$\height(i) = 1$ with exactly one leaf~$c$ having parent~$i$. For any other leaf~$v$ of~$\T$, we have~$\lca(c,v) \not\in \mathrm{Z}(\T)$. Hence, every vertex of~$\G$ is connected to vertex~$c$. This happens only when the block graph is a star graph. }
\end{proof}

\begin{remark} The correspondence between trees with colored and zeroed nodes and their BMT derived graphs is not one to one.  For example, the second graph in Figure \ref{fig:block_deleted} is also induced by the tree with leaves 1, 2, and 3 sharing the same least common ancestor and leaf 4 such that it is the only leaf with parent given by the top internal node.
\end{remark}

\subsection{\texorpdfstring{$G$}{G}-derived Laplacian transformations} \label{sec:g-derived-laplacian}
Recall that for a graph~$\Gamma$ on~$n$ vertices with weight $w_{ij}$ on edge~$\{i,j\}$, its Laplacian matrix is the~$n\times n$ square matrix where the off diagonal entry~$(i,j)$ is~$0$  when~$\{i,j\}\notin E(\Gamma)$, is~$-w_{ij}$ for~$(i,j)\in E(\Gamma)$, and its~$i$-th diagonal entry is~$\sum_{j:\{i,j\}\in E(\Gamma)}w_{ij}$.  

Consider graph~$G$ on vertex set~$[n]$ and edge set~$E(G)$. Construct the weighted complete graph~$\Gamma(G)$ on~$n+1$ vertices, with the extra node labeled~$0$ and with weights in the edges as follows:
\begin{enumerate}
    \item edge~$\{i,j\}$ such that~$i,j\neq 0$ and~$\{i,j\}\in E(G)$ has weight~$q_{ij}$,
    \item edge~$\{i,j\}$ such that~$i,j\neq 0$ and ~$\{i,j\}\not\in E(G)$  has weight~$-q_{ij}$,
    \item edge~$\{0,i\}$ has weight~$q_{0i} - \sum\limits_{j : \deg_G(j) = n-1} q_{0j}$ when~$\deg_G(i) < n-1$,
    \item edge~$\{0,i\}$ has weight~$q_{0i} - \sum\limits_{j:\deg_G(j)<n-1} q_{ij}$ when~$\deg_G(i) = n-1$.
\end{enumerate}
For an illustration see \Cref{fig:graph derived laplacian}. 
The \emph{$G$-derived Laplacian transformation} is the invertible linear transformation from the~$\sigma_{ij}$ variables to the~$q_{ij}$'s for~$0\leq i<j\leq n$ obtained by removing the first row and first column from the graph Laplacian matrix of~$\Gamma(G)$. 

When~$G$ is a complete graph, the~$G$-derived Laplacian transformation coincides with  the reduced graph Laplacian transformation in (\ref{eqn:graph_laplacian_BMT}). When~$G$ is a star block graph  with central node~$c$, this linear transformation is:
 \begin{align}\label{eqn:derived_laplacian}
 \begin{aligned}
        q_{ij} &= \begin{cases}
            -\sigma_{ij} & \textrm{if }\{i,j\}\in E(G),\\
            +\sigma_{ij} & \textrm{if }\{i,j\}\not\in E(G),
        \end{cases} \qquad 
        q_{0i} &=\begin{cases}
            +\sigma_{ii} & \textrm{if }i=c,\\
            \sum\limits_{j:j\not=c} \sigma_{ij} & \textrm{else}.
        \end{cases}
 \end{aligned}
    \end{align}

\begin{example}
Consider graph~$G$ and its induced weighted complete graph~$\Gamma(G)$ in \Cref{fig:graph derived laplacian}. 
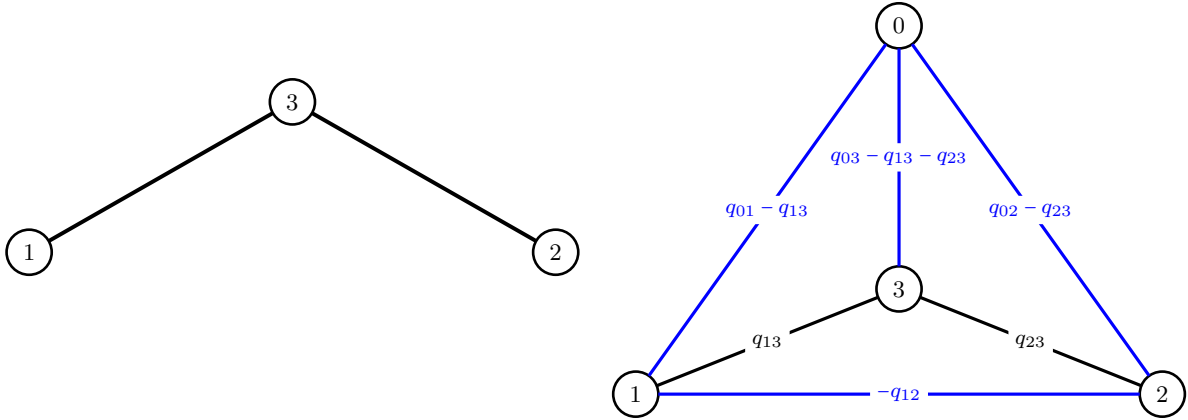
\begin{figure}[h]
    \centering
\begin{subfigure}{0.45\textwidth}
    \centering
     \begin{tikzpicture}
    \Vertex[x=0, label=$1$,color=white]{A}
    \Vertex[x=3, y=2,label=$3$,color=white]{B}
    \Vertex[x=6, label=$2$,color=white]{C}
    \Edge[color=black](A)(B)
    \Edge[color=black](B)(C)

    \node[draw, circle, white] (empty) at (0, -2) {};
\end{tikzpicture}
\end{subfigure}
\hspace{1em}
\begin{subfigure}{0.45\textwidth}
\centering
  \begin{tikzpicture}
  [scale=.7,very thick]
    \Vertex[x=0,y=0, label=$1$,color=white]{A}
    \Vertex[x=4, y=1.2,label=$3$,color=white]{B} 
    \Vertex[x=8, y=0, label=$2$,color=white]{C}
    \Vertex[x=4,y=6,label=$0$,color=white]{D}
    \draw[color=blue] (A) -- (D)
        node[midway, draw=white, rectangle, fill=white, text=blue, inner sep=2pt, font=\tiny] {$q_{01}-q_{13}$};
    \draw[color=blue] (B) -- (D)
        node[midway, draw=white, rectangle, fill=white, text=blue, inner sep=2pt, font=\tiny] {$q_{03}-q_{13}-q_{23}$};
    \draw[color=blue] (C) -- (D)
        node[midway, draw=white, rectangle, fill=white, text=blue, inner sep=2pt, font=\tiny] {$q_{02}-q_{23}$};
    \draw[color=black] (A) -- (B)
        node[midway, draw=white, rectangle, fill=white, text=black, inner sep=2pt, font=\tiny] {$q_{13}$};
    \draw[color=black] (B) -- (C)
        node[midway, draw=white, rectangle, fill=white, text=black, inner sep=2pt, font=\tiny] {$q_{23}$};
    \draw[color=blue] (A) -- (C)
        node[midway, draw=white, rectangle, fill=white, text=blue, inner sep=2pt, font=\tiny] {$-q_{12}$};
\end{tikzpicture}
\end{subfigure}
    \caption{A graph~$G$ and its induced weighted complete graph~$\Gamma(G)$.}
    \label{fig:graph derived laplacian}
\end{figure}

The graph Laplacian of~$\Gamma(G)$ is 
\begin{align*}
    \begin{bmatrix}
       q_{01}+q_{02}+q_{03} - 2(q_{13}+q_{23}) & -q_{01}+q_{13} & -q_{02}+q_{23} & -q_{03}+q_{13}+q_{23}\\
        -q_{01}+q_{13} & q_{01}-q_{12} & q_{12} & -q_{13}\\
        -q_{02}+q_{23} &  q_{12} & q_{02}-q_{12} & -q_{23}\\
        -q_{03}+q_{13}+q_{23} & -q_{13} &-q_{23} & q_{03}
    \end{bmatrix}.
\end{align*}
By removing the first row and the first column we obtain the~$G$-derived Laplacian transformation.
\end{example}

Although in this paper we only focus on applying the~$G$-derived Laplacian to obtain toric ideals,  combinatorial properties and other applications of this transformation would be interesting to study.

\subsection{BMT-derived block graphs are toric}

Before we state the main result of the section, note that the \emph{generalized path map}  (\ref{eqn:generalized path map}) on the colored tree~$\T$ with zeroed nodes does not {parameterize the variety $L_{\mathcal{T}}$} when~$Z(\T)\neq \emptyset$. So we make the following adjustment. Let~$c$ be the unique non-root leaf connected to the top internal node. Consider~$\varphi'_\T$ on~$\T$ such that  
\begin{equation} \label{eqn:generalized path map block}
    \begin{array}{cccc}
         \varphi'_\T: &\C[q_{ij} \mid 0\leq i<j\leq n] &\to &\C[\theta_{k} \mid k\in \Lv(T)\cup \Int(T)] \\
         & q_{ij} & \mapsto & \begin{cases}
             \theta_{{c}}^2 & \text{if } (i,j) = (0,c)\\
             \prod\limits_{\substack{(\ell, k)\in i \leftrightsquigarrow j\\ k\notin Z(\T)}} \theta_{k} & \text{else}.
         \end{cases}
    \end{array}
\end{equation}

We also consider embeddings~$\Tilde{\I}_G$ and~$\Tilde{I}_T$ of ideals~$\I_G$ and~$I_T$, respectively, in the polynomial ring $\C[Q]\coloneqq \C[q_{ij}\mid 0\leq i<j\leq n]$ as follows.  Let~$\Tilde{\I}_G$ be the ideal in~$\C[Q]$ obtained by substitution of variables in the ideal~$\I_G\subset \C[\Sigma]$ in \Cref{thm:misra_sullivant}:~$ \sigma_{ii}=q_{0i} \text{ and }  \sigma_{ij}=q_{ij} \text{ for  } 1\leq i<j\leq n.$ 
Similarly, let~$\Tilde{I}_T$ be the ideal in~$\C[Q]$ obtained by substitution of variables in the ideal~$I_T\subset \C[p_{ij} \mid 0\leq i<j\leq n]$ in \Cref{thm:brownian_toric}:~$p_{0i}=q_{0i} \text{ and }  p_{ij}=q_{ij} \text{ for  } 1\leq i<j\leq n.$ Now we are ready to state the main result of the section.

\begin{theorem}\label{thm:deletions}
Let~$\G =(G,\Lambda_{\G})$ be a BMT-derived block graph from~$\T = (T,\Lambda_\T,\mathrm{Z}(\T))$ with~$\mathrm{Z}(\T)\not=\emptyset$ {such that all nodes of~$\T$ have distinct colors}. Then, the vanishing ideal of~$L_\T^{-1}=\L_\G^{-1}$ under the~$G$-derived Laplacian transformation is the toric ideal 
$\sqrt{\Tilde{I}_T+ \Tilde{\I}_{G}}=\ker \varphi'_\T$.  
\end{theorem}

\begin{proof}
We have~$\L_{\G} = L_T\cap \L_{G}$, so by \Cref{thm:ACTUAL_INTERSECTION},  it suffices to show that the vanishing ideals of~$\L_G^{-1}$ and of~$L_T^{-1}$ are binomial under the~$G$-derived Laplacian.  Let~$\psi$ represent the~$G$-derived Laplacian transformation.

Ideal~$\I_{G}$ in \Cref{thm:misra_sullivant} is the vanishing ideal of~$\L_G^{-1}$.   
By \Cref{prop:star_structure},~$G$ has a central vertex, say~$c$. As such, by Theorem \ref{thm:misra_sullivant},~$\I_{G}$ is generated by binomials~$\sigma_{ik}\sigma_{j\ell} -   \sigma_{i\ell}\sigma_{jk}$ for distinct vertices~$i,j\in A\cup \{c\},k,\ell\in B\cup \{c\}$,  and~$\sigma_{cc}\sigma_{j\ell} -   \sigma_{c\ell}\sigma_{jc}$ for~$j \in A$,~$\ell \in B$ for all possible 1-clique partitions~$(A,B,\{c\})$ of~$G$.
\Cref{fig:subgraph_structures} shows the possible  subgraphs up to relabeling, when~$i,j,k,\ell$ are distinct.
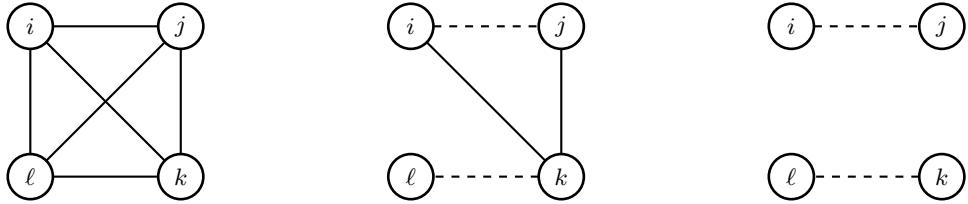
\begin{figure}[h]
\centering
    \begin{minipage}{0.3\textwidth}
    \centering
         \begin{tikzpicture}
    \Vertex[label=$i$,color=white]{i}
    \Vertex[x=2,label=$j$,color=white]{j}
    \Vertex[x=2,y=-2, label=$k$,color=white]{k}
    \Vertex[y=-2,label=$\ell$,color=white]{l}
    \draw[color=black,thick](i)--(j);
    \draw[color=black,thick](i)--(k);
    \draw[color=black,thick](i)--(l);
    \draw[color=black,thick](j)--(k);
    \draw[color=black,thick](j)--(l);
    \draw[color=black,thick](k) --(l);
\end{tikzpicture}
    \end{minipage}
    \begin{minipage}{0.3\textwidth}
    \centering
             \begin{tikzpicture}
    \Vertex[label=$i$,color=white]{i}
    \Vertex[x=2,label=$j$,color=white]{j}
    \Vertex[x=2,y=-2, label=$k$,color=white]{k}
    \Vertex[y=-2,label=$\ell$,color=white]{l}
    \draw[color=black,thick,dashed](i)--(j);
    \draw[color=black,thick](i)--(k);
    \draw[color=black,thick](j)--(k);
    \draw[color=black,thick,dashed](k) --(l);
\end{tikzpicture}
    \end{minipage}
\begin{minipage}{0.3\textwidth}
\centering
     \begin{tikzpicture}
    \Vertex[label=$i$,color=white]{i}
    \Vertex[x=2,label=$j$,color=white]{j}
    \Vertex[x=2,y=-2, label=$k$,color=white]{k}
    \Vertex[y=-2,label=$\ell$,color=white]{l}
    \draw[color=black,thick,dashed](i)--(j);
    \draw[color=black,thick,dashed](k) --(l);
\end{tikzpicture}
\end{minipage}
    \caption{Possible block subgraphs on 4 distinct vertices, dashed lines represent edges whose presence does not change the generating binomial.}
    \label{fig:subgraph_structures}
\end{figure}

First, assume without loss of generality that~$i=k=c$, where~$c$ is the central vertex. Note that~$\psi(\sigma_{cc}) = q_{0c}$. Then,~$\{i,\ell\}\{j,k\}\in E(G)$. Since~$j\in A,\ell\in B$, we have~$\{j,\ell\}\not\in E(G)$. As such,~$$\psi(\sigma_{ik}\sigma_{j\ell}-\sigma_{i\ell}\sigma_{jk}) = q_{0c}q_{j\ell} - (-q_{c\ell})(-q_{cj}) = q_{0c}q_{j\ell} - q_{c\ell}q_{cj}.$$

Otherwise,~$i,j,k,\ell$ are all distinct. Due to the partition structure,~$i,j,k,\ell$ cannot form a complete subgraph. As such, up to relabeling, we have the following two cases, visualized in \Cref{fig:subgraph_structures}(b,c):
\begin{enumerate}
    \item $k$ is connected to~$i$ and~$j$, with~$\ell$ disconnected from~$i$ and~$j$. Then, $\{i,k\},\{j,k\}\in E(G)$ and $\{j,\ell\},\{i,\ell\}\not\in E(G)$, so~$\psi(\sigma_{ik}\sigma_{j\ell} - \sigma_{i\ell}\sigma_{jk}) = -q_{ik}q_{j\ell} + q_{i\ell}q_{jk}$.
    \item $\{i,k\},\{j,\ell\},\{i,\ell\},\{j,k\}\not\in E(G)$. Then,~$\psi(\sigma_{ik}\sigma_{j\ell} - \sigma_{i\ell}\sigma_{jk}) = q_{ik}q_{j\ell} - q_{i\ell}q_{jk}$.
\end{enumerate}

This shows that the image of~$\I_G$ under~$\psi$ is a binomial ideal and this image is precisely~$\Tilde{\I}_G$.

Next, let~$I_T$ be the vanishing ideal of~$L_T$ under the graph Laplacian transformation. By \Cref{thm:brownian_toric}, the generators of~${I}_T$ are given by~$p_{ik}p_{j\ell} - p_{i\ell}p_{jk}$ for cherries~$\{i,j\},\{k,\ell\}$ of~$T$ in the induced  4-leaf sub-tree on~$i,j,k,\ell$.  We consider the generating set of the vanishing ideal for~$L_T^{-1}$ obtained by the inverse reduced Laplacian on these binomial generators. 
When~$i,j,k,\ell\not=0$, this binomial expressed in the~$\sigma_{ij}$'s is
\(
 \sigma_{ik}\sigma_{j\ell}-\sigma_{i\ell}\sigma_{jk}.\)
Recall that~$\psi(\sigma_{ik}) = -q_{ik}$ if~$\{i,j\}\in E(G)$, and~$\psi(\sigma_{ik}) = q_{ik}$ otherwise. The same holds for any other pair of non-root leaves. Examining the possible subgraphs (see \Cref{fig:subgraph_structures}), all cases give \(\psi(\sigma_{ik}\sigma_{j\ell} - \sigma_{i\ell}\sigma_{jk}) = \pm (q_{ik}q_{j\ell} - q_{i\ell}q_{jk}).\)

Without loss of generality  assume~$\ell=0$ is the root leaf. Then, the generator~$p_{ik}p_{0j} - p_{0i} p_{jk}$ expressed in the~$\sigma_{ij}$'s is
    \begin{align*}
        \sigma_{ik}\left(\sum_{s=1}^n \sigma_{js} \right) - \left(\sum_{s=1}^n \sigma_{is} \right)\sigma_{jk}.
    \end{align*}

Note that~$i,j$ cannot be equal to~$c$. As such, we always have~$q_{0i} = \sum\limits_{s\not=c} \sigma_{si}$ and~$q_{0j} = \sum\limits_{s\not=c} \sigma_{sj}$, and so
    \begin{align*}
    \begin{aligned}
       &\psi\left(\sigma_{ik}\left(\sum_{s=1}^n \sigma_{js} \right) - \left(\sum_{s=1}^n \sigma_{is} \right)\sigma_{jk}\right) =
       \\&= \psi\left(\sigma_{ik}\right)\big(q_{0j}+ \psi(\sigma_{cj})\big) - \big(q_{0i} + \psi(\sigma_{ci})\big)\psi(\sigma_{jk})=
       \\
       &=\Big( \psi(\sigma_{ik})q_{0j} - q_{0i}\psi(\sigma_{jk}) \Big) + \Big(\psi(\sigma_{ik})\psi(\sigma_{cj}) - \psi(\sigma_{ci}) \psi(\sigma_{jk})\Big).\end{aligned}
    \end{align*}
The subtree on leaves $i,j,k,c$ will have cherries~$\{i,j\},\{c,k\}$ which gives $\psi(\sigma_{ik}) \psi(\sigma_{cj}) - \psi(\sigma_{ci})\psi(\sigma_{jk}) = q_{ik}q_{cj}-q_{ci}q_{jk}$, 
shown earlier to be in the image of the vanishing ideal of~$L_T^{-1}$ under~$\psi$. As such, in the image,  we can replace the generator~$\psi\left(\sigma_{ik}\left(\sum_{s=1}^n \sigma_{js} \right) - \left(\sum_{s=1}^n \sigma_{is} \right)\sigma_{jk}\right)$ with the binomial~$ \psi(\sigma_{ik})q_{0j} - q_{0i}\psi(\sigma_{jk})=\pm (q_{ik}q_{0j} - q_{0i}q_{jk})$. 
Note that the binomial generators are generators for the toric ideal~$\Tilde{I}_T$, which makes the latter the vanishing ideal of~$L_T$ under~$\psi$. \Cref{thm:ACTUAL_INTERSECTION} concludes that~$ \sqrt{\Tilde{I}_T + \Tilde{\I}_G}$ the toric vanishing ideal of~$L_\T=\L_\G$ under the G-derived Laplacian.

Finally, observe that any generator of~$\Tilde{I}_T + \Tilde{\I}_{G}$ maps to zero under~$\varphi'_\T$.  Indeed, we would only need to check generators that include~$q_{0c}$. Given that~$c$ corresponds to the only non-root leaf, the subtree on~$i,j,0,c$ will always have cherry structure~$\{i,j\},\{0,c\}$. As such, there is no generator containing~$q_{0c}$ in~$\Tilde{I}_T$. Regarding~$\Tilde{\I}_\G$, generators containing~$p_{0c}$ are of the form~$q_{0c}q_{ij}-q_{ic}q_{jc}$ for~$i,j\neq 0,c$. With the adaptation we made to the generalized path map, this is mapped to zero by~$\varphi'_\T$.
So,~$I_T + \I_{G}\subseteq \ker \varphi'_\T$. Taking radicals both sides one has~$\sqrt{I_T + \I_{G}}\subseteq \ker \varphi'_\T$. Now~$\dim(\ker \varphi'_\T)=\dim\left(\sqrt{I_T + \I_{\G}}\right)=|V(T)|-|Z(\T)|$. Together with the fact that a prime ideal cannot strictly contain other prime ideals of the same dimension imply that~$\sqrt{I_T + \I_{G}}= \ker \varphi'_\T$.
    \end{proof}

\begin{remark}
To our surprise,  non-block colored graphs are sometimes toric. For example, the graph in \Cref{fig:non_block_toric}, which is not toric under the~$G$-derived Laplacian, is toric under the change of variables given in (\ref{eq:nonblocktoric transform}). We found this transformation using the code from \cite{kahle2024lie}. This is perhaps unexpected, since no uncolored non-block graphs  on 4 and 5 vertices has toric structure \cite{kahle2024lie}.

\begin{align}\label{eq:nonblocktoric transform}
  \begin{pmatrix}
        p_{11}\\
        p_{12}\\
        p_{13}\\
        p_{14}\\
        p_{22}\\
        p_{23}\\
        p_{24}\\
        p_{33}\\
        p_{34}\\
        p_{44}
    \end{pmatrix} &= \begin{pmatrix}
0 & 1 & 0 & 0 & 0 & 1 & 0 & 1 & 0 & 0 \\
0 & 0 & 1 & 0 & 0 & 0 & 1 & 0 & 1 & 0 \\
1 & 0 & 0 & 1 & 0 & -\frac{12}{5} & -\frac{12}{5} & -3 & -3 & 0 \\
\frac{1}{2} & 0 & 0 & -1 & 0 & -\frac{6}{5} & -\frac{6}{5} & 3 & 3 & 0 \\
0 & 0 & 0 & 0 & 0 & 0 & 0 & 0 & 0 & 1 \\
0 & 0 & 0 & 0 & 0 & \frac{12}{5} & 0 & 3 & 0 & 0 \\
0 & 0 & 0 & 0 & 0 & \frac{6}{5} & 0 & -3 & 0 & 0 \\
0 & 0 & 0 & 0 & 0 & 0 & \frac{12}{5} & 0 & 3 & 0 \\
0 & 0 & 0 & 0 & 0 & 0 & \frac{6}{5} & 0 & -3 & 0 \\
\frac{6}{17} & 1 & -1  & 0 & 1 & -\frac{37}{20}  & -\frac{37}{25}  & -2  & 2  & -2  \\
\end{pmatrix}
    \begin{pmatrix}
        \sigma_{11}\\
        \sigma_{12}\\
        \sigma_{13}\\
        \sigma_{14}\\
        \sigma_{22}\\
        \sigma_{23}\\
        \sigma_{24}\\
        \sigma_{33}\\
        \sigma_{34}\\
        \sigma_{44}
    \end{pmatrix} 
    \end{align}
\end{remark}

\begin{figure}[h]
    \centering
\begin{subfigure}[b]{0.3\textwidth}
    \centering
    \begin{tikzpicture}[scale=0.4]
            \tikzset{
                VertexStyle/.style={
                    shape=circle,
                    draw,
                    fill=white,
                    minimum size=10pt,
                    inner sep=0pt,
                    font=\tiny
                },
                EdgeStyle/.style={
                    color=black,
                    line width=0.5pt,
                    draw=none, fill=none, midway
                }
            }
            \Vertex[x=0, y=5.5, label=0, color=gray!0, style={opacity=0}]{n0}
            \Vertex[x=-1, y=0, label=2, color=yellow]{n2}
            \Vertex[x=-3, y=0, label=1, color=cyan]{n1}
            \Vertex[x=1, y=0, label=3, color=magenta]{n3}
            \Vertex[x=3, y=0, label=4, color=violet!50]{n4}
            \Vertex[x=-2, y=2, label=\textcolor{black}{5}, color=white]{n5}
            \Vertex[x=2, y=2, label=\textcolor{white}{6}, color=darkgreen]{n6}
            \Vertex[x=0, y=4, label=\textcolor{white}{7}, color=blue]{n7}

            \draw (n0) -- (n7);
            \draw (n7) -- (n6);
            \draw (n6) -- (n4);
            \draw (n7) -- (n5);
            \draw (n5) -- (n2);
            \draw (n5) -- (n1);
            \draw (n6) -- (n3);
        \end{tikzpicture}
\end{subfigure}
\begin{subfigure}[b]{0.3\textwidth}
\begin{tikzpicture}
    \Vertex[label=$1$,color=cyan]{A}
    \Vertex[x=2,label=$2$,color=yellow]{B}
    \Vertex[y=-2,label=$4$,color=violet!50]{D}
    \Vertex[x=2,y=-2, label=$3$,color=magenta]{C}
    \Edge[color=blue](A)(D)
    \Edge[color=blue](B)(D)
    \Edge[color=green](C)(D)
    \Edge[color=blue](A)(C)
    \Edge[color=blue](B)(C)
\end{tikzpicture}
\end{subfigure}
    \caption{A colored tree with zeroed node~$5$ with non-block BMT derived graph whose vanishing ideal is toric under transformation (\ref{eq:nonblocktoric transform}).}
    \label{fig:non_block_toric}
\end{figure}
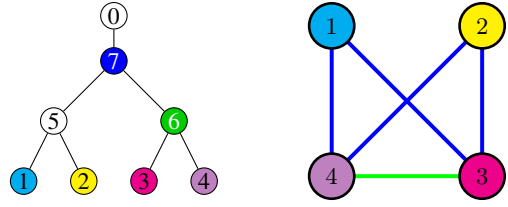

\section{Main result and discussion}\label{sec:6}
Now we can combine \Cref{thm:vertex_regular} and \Cref{thm:deletions} to extend the class of trees with colored and zeroed nodes with toric reciprocal linear space. First, let us introduce some notation. 
Given~$\T$ with BMT derived graph~$\G$, let~$\tilde{\I}_{\overline{\G}}$ be the ideal in~$\C[q_{ij} \mid 0\leq i<j\leq n]$ obtained by substitution of variables in the ideal in \Cref{lemma:jensen_corollary} as follows:~$\sigma_{ii} = q_{0i}\textrm{ and }\sigma_{ij} = q_{ij} \textrm{ for }1\leq i< j \leq n.$

Let the generalized path map for a colored tree~$\T$ with~$Z(T)\neq \emptyset$ be 
\[
\begin{array}{cccl}
     \varphi'_\T: &\C[q_{ij} \mid 0\leq i<j\leq n] &\to &\C[\theta_{\lambda} \mid \lambda\in \Lambda_\T]\\
     &q_{ij} & \mapsto & \begin{cases}
         \theta_{{\lambda(c)}}^2 & \text{if } (i,j) = (0,c) \\
         \prod\limits_{\substack{(\ell, k)\in i \leftrightsquigarrow j\\ k\notin Z(\T)}} \theta_{\lambda(k)} & \text{else}.
     \end{cases}
\end{array}
\]

\begin{theorem}
Let~$\G = (G,\Lambda_\G)$ be a BMT-derived graph from tree~$\T = (T,\Lambda_T, Z(\T))$. Suppose that~$\G$ is a vertex-regular block graph {and let~$\overline{\G}$ be the vertex-regular completion of~$\G$}. Then, the vanishing ideal of~$L_\T^{-1}=\L_\G^{-1}$ under the~$G$-derived Laplacian is~$ \sqrt{\tilde{I}_T + \tilde{\I}_G + \tilde{\I}_{\overline{\G}}}=\ker \varphi'_\T$.
\end{theorem}
\begin{proof}
Let~$\psi$ represent the~$G$-derived Laplacian transformation. We have $\L_{\G} = L_{T}\cap \L_{\G} \cap \L_{\overline{\G}}$. In \Cref{thm:deletions}, we showed that the vanishing ideals of~$L_T^{-1}$ and~$\L_{G}^{-1}$ are the toric ideals~$\Tilde{I}_T$ and~$\Tilde{\I}_G$, respectively, under~$\psi$. {As such, it suffices to show that $\psi(I)$ is binomial, where $I$ is the the vanishing ideal of the vertex-regular completion, ~$\L_{\overline{\G}}^{-1}$.} Recall from \Cref{lemma:jensen_corollary} that the vanishing ideal of~$\L_{\overline{\G}}^{-1}$ is given by:~$$I = \left\langle  \sigma_{ik} - \sigma_{jk}, \ \sigma_{ii}-\sigma_{jj} \mid  i,j \in [n], \lambda_{\overline{G}}(i) = \lambda_{\overline{G}}(j) \text{ in } \overline{\G}, k\in  [n]\setminus \{i,j\}\right\rangle.$$
\Cref{prop:vertex_reg_iff} implies~$\pa(i) = \pa(j)$, so~$\lca(i,k) = \lca(j,k)$ for any~$k\in [n]\setminus\{i,j\}$. By \Cref{prop:star_structure} there is only one leaf with parent equal to the top internal node, so~$\pa(i),\pa(j)\not=c$. 

First, for~$k\in [n]\setminus \{i,j\}$, we get~$\psi(\sigma_{ik} - \sigma_{jk}) = q_{ik}-q_{jk}$ if~$\lca(i,k)\not\in \mathrm{Z}(\T)$, and~$\psi(\sigma_{ik} - \sigma_{jk}) = -q_{ik}+q_{jk}$ if~$\lca(i,k)\in \mathrm{Z}(\T)$.
For the remaining  generators of~$I$ we obtain
\begin{align*}
    \psi(\sigma_{ii} - \sigma_{jj}) &= \psi\left(\sum_{k\not=c}\sigma_{ik} - \sum_{k\not=c}\sigma_{jk}\right) - \psi\left(\sum_{k\not=c,i,j} (\sigma_{ik} - \sigma_{jk})\right) - \psi(\sigma_{ij}- \sigma_{ij})\\
    &= q_{0i}-q_{0j} - \sum_{k\not=c,i,j}\pm (q_{ik}-q_{jk}).
\end{align*}
We see that the sum in the expression above is already in the ideal, as we have~$q_{ki} - q_{kj} \in \psi(I)$ for all~$k\in [n]\setminus \{ i,j\}$. This shows~$\psi(I)$ is the binomial ideal~$ \tilde{\I}_{\overline{\G}}$. From \Cref{thm:deletions} and  \Cref{thm:ACTUAL_INTERSECTION}, it follows that the variety of covariance matrices of~${\G}$, in the star graph Laplacian coordinates, coincides with the toric variety defined by~$\sqrt{\tilde{I}_T+ \tilde{\I}_{G} + \tilde{\I}_{\overline{\G}}}$. The proof for~$\sqrt{\tilde{I}_T+ \tilde{\I}_{G} + \tilde{\I}_{\overline{\G}}}=\ker\varphi_\T'$ is analogous to the last paragraph of the proof of \Cref{thm:deletions}. Generators of ideals~$\tilde{I}_T,\tilde{\I}_{G},\tilde{\I}_{\overline{\G}}$ are mapped to zero by~$\varphi_\T'$, the two ideals have the same dimension equal to the number of colors in the tree, thus the prime ideals are equal.
\end{proof}

\section{Toward phylogenetic networks} 
Phylogenetic trees with colored and zeroed nodes are a step toward a more accurate model for colored phylogenetic networks with species hybridization. As an example, let us revisit the tree~$\T$ from \Cref{fig:intro}.  The network~$\mathcal{N}$ on the right of \Cref{fig:discussion} presents a model corresponding to this tree. Here, node $h$ represents the hybridization of nodes $7'$ and $7''$, both sharing the same color as~$7$. Following the continuous interpretation of the Gaussian structural equation model for trees \cite{sturmfels2019brownian}, we associate to each colored node~$i$ in the network a Gaussian random variable~$\varepsilon_{\lambda(i)}$ with mean~$0$ and variance~$\tau_{\lambda(i)}\geq 0$. 
We define the Markov process on~$\mathcal{N}$ as follows. For each node~$i\neq h$ in~$\mathcal{N}$, starting from the root, inductively  
define random variable~$Y_i = Y_j+\varepsilon_i$, where~$j$ is the parent of~$i$,~$Y_0=0$, and~$Y_h=Y_{7'}+Y_{7''}$.  The parameter~$\tau_\lambda$ can be interpreted as the length of the edges~$\{i,j\}$ with~$\lambda(i)=\lambda$, and it can be related to the parameter~$\theta_\lambda$ that appears in the generalized path map (\ref{eqn:generalized path map}) as~$\tau_\lambda = e^{\theta_\lambda}$. This construction yields a set of multivariate Gaussian random variables that share the same symmetries in their covariance matrix as those obtained from the tree on the left. 

\begin{figure}[ht]
\centering
\begin{subfigure}[b]{0.3\textwidth}
    \centering
        \begin{tikzpicture}[scale=0.4]
            \tikzset{
                VertexStyle/.style={
                    shape=circle,
                    draw,
                    fill=white,
                    minimum size=10pt,
                    inner sep=0pt,
                    font=\tiny
                },
                EdgeStyle/.style={
                    color=black,
                    line width=0.5pt
                }
            }
            \Vertex[x=0, y=-0.5, label=0, color=white, style={opacity=0}]{n0}
            \Vertex[x=-1, y=-6, label=\textcolor{black}{2}, color=cyan]{n2}
            \Vertex[x=-3, y=-6, label=\textcolor{black}{1}, color=cyan]{n1}
            \Vertex[x=1, y=-6, label=3, color=yellow]{n3}
            \Vertex[x=3, y=-6, label=\textcolor{white}{4}, color=darkgreen]{n4}
            \Vertex[x=-2, y=-4.66, label=\textcolor{white}{5}, color=red]{n5}
            \Vertex[x=-1, y=-3.33, label=6, color=white]{n6}
            \Vertex[x=0, y=-2, label=\textcolor{white}{7}, color=blue]{n7}

            \Edge[color=black](n0)(n7)
            \Edge[color=black](n7)(n6)
            \Edge[color=black](n6)(n5)
            \Edge[color=black](n7)(n4)
            \Edge[color=black](n5)(n2)
            \Edge[color=black](n5)(n1)
            \Edge[color=black](n6)(n3)
        \end{tikzpicture}
\end{subfigure}
\begin{subfigure}[b]{0.3\textwidth}
    \centering
    \begin{tikzpicture}[scale=0.4]
            \tikzset{
                VertexStyle/.style={
                    shape=circle,
                    draw,
                    fill=white,
                    minimum size=10pt,
                    inner sep=0pt,
                    font=\tiny
                },
                EdgeStyle/.style={
                    color=black,
                    line width=0.5pt,
                    draw=none, fill=none, midway
                }
            }
            \Vertex[x=0, y=5.5, label=0, color=gray!0, style={opacity=0}]{n0}
            \Vertex[x=-2, y=0, label=2, color=cyan]{n2}
            \Vertex[x=-4, y=0, label=1, color=cyan]{n1}
            \Vertex[x=0, y=0, label=\textcolor{white}{4}, color=darkgreen]{n4}
            \Vertex[x=3, y=0, label=3, color=yellow]{n3}
            \Vertex[x=-3, y=2, label=\textcolor{white}{5}, color=red]{n5}
            \Vertex[x=-2, y=3.75, label=\textcolor{white}{$7'$}, color=blue]{n7}
            \Vertex[x=0, y=3.75, label=\textcolor{white}{$h$}, color=gray]{n8}
            \Vertex[x=2, y=3.75, label=\textcolor{white}{$7''$}, color=blue]{n7'}

            \draw (n0) -- (n7') [black] node [pos=0.3, right, black] {\tiny~$\tau_{b}$};
            \draw (n0) -- (n7) [black] node [pos=0.3, left, black] {\tiny~$\tau_{b}$};
            \draw (n7) -- (n5) [black] node [pos=0.3, left, black] {\tiny~$\tau_{r}$};
            \draw[dashed] (n7') -- (n8) ;
            \draw (n7') -- (n3) [black] node [pos=0.5, right, black] {\tiny~$\tau_{y}$};
            \draw[dashed] (n8) -- (n7) ;
            \draw (n5) -- (n2) [black] node [pos=0.3, right, black] {\tiny~$\tau_{c}$};
            \draw (n5) -- (n1) [black] node [pos=0.3, left, black] {\tiny~$\tau_{c}$};
            \draw (n8) -- (n4) [black] node [pos=0.5, right, black] {\tiny~$\tau_{g}$};
        \end{tikzpicture}
\end{subfigure}
\caption{Tree with colored and zeroed nodes (left) and its corresponding phylogenetic network (right). Node~$h$ is a ``hybridization" of nodes~$7'$ and~$7''$. The tree and the network induce the same symmetries in the covariances between leaves.}
\label{fig:discussion}
\end{figure}

\section*{Acknowledgments}
We thank Elizabeth Allman, Jack Jeffries, Pablo Maz\'on, John Rhodes, and Anna Seigal for helpful discussions, {and the referee for helpful comments}. EC and AM were supported by the National Science Foundation under Grant DMS-2306672. AR was supported by fellowships from ''la
Caixa'' Foundation (ID 100010434), with code LCF/BQ/EU23/12010097, and from RCCHU. Most of the work was done in Pierce Hall 212A (SEAS) at Harvard University. 

\bibliographystyle{plain}
\bibliography{references}

\end{document}